\documentclass[12pt]{article}
\usepackage{verbatim}
\usepackage{fancyhdr}
\usepackage{graphicx}
\usepackage{epstopdf}
\usepackage{amsmath,amsfonts,amsthm,amssymb}
\usepackage{subfigure}
\usepackage{color}
\usepackage{bm}
\usepackage{comment}
\usepackage{multirow}
\usepackage{enumerate}
\usepackage{url}

\newtheorem{lem}{Lemma}[section]
\newtheorem{thm}{Theorem}[section]

\newtheorem{rem}{Remark}[section]
\newtheorem{exmp}{Example}

\numberwithin{equation}{section}

\newcommand{\jl}{j-\frac 12}
\newcommand{\jr}{j+\frac 12}
\newcommand{\il}{i-\frac 12}
\newcommand{\ir}{i+\frac 12}
\newcommand{\g}{\mathbb{G}_{\theta_1,\theta_2}}
\newcommand{\gr}{\widetilde{\mathbb{G}}_{\theta_1,\frac{1}{2}}}
\newcommand{\gl}{\widetilde{\mathbb{G}}_{\frac{1}{2},\theta_2}}

\newcommand{\jpl}{[\![}
\newcommand{\jpr}{]\!]}
\providecommand{\mean}[1]{ \{\mspace{-6.0mu}\{ #1 \}\mspace{-6.0mu}\} }
\providecommand{\jump}[1]{ [\mspace{-2.5mu}[ #1 ]\mspace{-2.5mu}] }

\begin{document}

\renewcommand\baselinestretch{1.24}

\title{An oscillation free local discontinuous Galerkin method for nonlinear degenerate parabolic equations}
\date{}
\author{Qi Tao\footnote {Beijing Computational Science Research Center, Beijing 100193, China.
E-mail: taoqi@csrc.ac.cn. Research is supported in part by NSFC grants U1930402 and the fellowship of China
Postdoctoral Science Foundation No. 2020TQ0030},~~
Yong Liu \footnote {
LSEC, Institute of Computational Mathematics, Hua Loo-Keng Center for Mathematical Sciences,
Academy of Mathematics and Systems Science, Chinese Academy of Sciences, Beijing 100190, China.
E-mail: yongliu@lsec.cc.ac.cn. Research is partially supported by the fellowship of China Postdoctoral Science
Foundation No. 2020TQ0343.},~~
Yan Jiang \footnote {School of Mathematical Sciences, University of Science and Technology of China,
Hefei, Anhui 233026, China. E-mail: jiangy@ustc.edu.cn. Research is partially supported by NSFC grant 11901555},~~
Jianfang Lu\footnote {South China Research Center for Applied Mathematics
and Interdisciplinary Studies, South China Normal University, Canton,
Guangdong 510631, China. E-mail: jflu@m.scnu.edu.cn. Research is partially supported by NSFC grant 11901213
and Guangdong Basic and Applied Basic Research Foundation 2020B1515310021. }
}

\maketitle

\textbf{Abstract.}
In this paper, we develop an oscillation free local discontinuous Galerkin (OFLDG) method for solving nonlinear degenerate parabolic equations. Following the idea of our recent work \cite{LLS2021SINUM}, we add the damping terms to the LDG scheme to control the spurious oscillations when solutions have a large gradient. The $L^2$-stability and optimal priori error estimates for the semi-discrete scheme are established. The numerical experiments demonstrate that the proposed method maintains the high-order accuracy and controls the spurious oscillations well.
\medskip

\medskip


\section{Introduction}
\label{sec_intro}

In this paper, we are interested in designing an oscillation free local discontinuous
Galerkin (OFLDG) method for solving the nonlinear degenerate parabolic equations in the following form:
\begin{align} \label{eqn_cdp}
 u_t + \nabla \cdot \big( \bm{f}(u) - \bm{a}(u) \, \nabla u \big) = 0,\quad  \bm{a}(u)\in \mathbb{R}^{d\times d},
\end{align}
where $\bm{x} = ( x_1, \ldots, x_d)^T\in\Omega$ and $\Omega  \subseteq \mathbb{R}^{d}$ is open
and bounded. The flux functions $\bm{f}(u) = \big( f_1(u), \ldots, f_d(u) \big)^T$ and
$\bm{a}(u)=[a_{ij}(u)]_{d\times d}$ is a positive semidefinite matrix.
In particular, when $\bm{a}(u) = 0$, the scalar hyperbolic conservation laws are served as
the special cases of \eqref{eqn_cdp}. It also includes the heat equation and the porous
medium type equations which are often termed as
{\it degenerate parabolic equations} (DPEs),
that is, $\bm{a}(u)$ vanishes for some certain values of $u$.
Consequently, the partial differential equations of type \eqref{eqn_cdp} model
a wide range of phenomena, such as porous media flow \cite{Aronson1986},
glacier movement and growth \cite{KL1999IMAJNA} and sedimentation processes
\cite{BCBT1999}, etc.
For the non-degenerate ($\bm{a}(u)\neq0$ for all $u$) problem \eqref{eqn_cdp}, it is widely known
that it admits a unique classic solution. However, when $\bm{a}(u) = 0$ for some $u$, the solution
may not be smooth anymore due to the hyperbolic nature of \eqref{eqn_cdp}. Some theoretical results on
the existence and uniqueness of the solution to \eqref{eqn_cdp} can be found in
e.g. \cite{AL1983MathZ, DiBenedetto1993, KR2003DCDS} and the references therein.
In the degenerate case, the solution is often non-smooth and we have to seek
a weak solution. The low regularity of the solution also
brings difficulties to the numerical simulation, especially for the high order methods.
In fact, the spurious oscillations may occur near the interfaces and wave fronts
that are harmful to the robustness of the numerical algorithm.
To overcome this difficulty, various schemes and approaches have been developed in
the literature, such as interface tracking algorithms \cite{DH1984TAMS},
diffusive kinetic schemes \cite{DNT2004MC}, relaxation
schemes \cite{CNPS2007SINUM}, finite difference/volume weighted essentially non-oscillatory (WENO) methods \cite{AHZ2019JCP, Jiang2021JSC,Liu2011SISC},
entropy stable schemes with artificial viscosity \cite{Jerez2017sinum},
method of lines transpose (${\rm MOL}^{T}$) approach with nonlinear filters \cite{CGJY2020JSC},
 discontinuous Galerkin (DG) methods with maximum-principle-satisfying limiters \cite{SCS2018JCP,ZZS2013JCP},
local DG finite element methods \cite{ZW2009JSC}, direct DG methods \cite{Liu2009sinum}, etc.

In this paper, we focus on the DG method and extend our previous
work \cite{LLS2021SINUM} to the nonlinear convection-diffusion problem \eqref{eqn_cdp}.
Compared with the continuous finite element method, the DG method has its own
advantages such as the allowance of the hanging nodes, easy $h$-$p$ adaptivity,
and high parallel efficiency because of the extremely local data structure.
The first DG method was introduced by
Reed and Hill to solve a steady linear transport problem \cite{RH1973} in 1973.
Later on, Cockburn et al.
combined the DG discretization in space with the Runge-Kutta time discretization method to
solve the hyperbolic conservation laws successfully in a series of papers
\cite{ CHS1990MC, CLS1989JCP, CS1989MC, CS1991M2AN, CS1998JCP}.
Enlightened by \cite{BR1997JCP,BRSMP1995},
Cockburn and Shu developed the local
discontinuous Galerkin (LDG) method to solve the convection-diffusion equations in \cite{CS1998LDG}.
Since the solution of \eqref{eqn_cdp} may not have enough regularity, the DG method
becomes a natural choice for its ability to deal with non-smooth solutions.
Conventionally, there are two approaches to deal with the
spurious oscillations in the DG method. One is to apply the slope limiters to the numerical
solutions at each time level to make them meet specific needs.
There exist many effective and efficient limiters such as the $minmod$ type total variation
diminishing (TVD) limiter, total variation bounded (TVB) limiter, weighted essentially
non-oscillatory (WENO) limiter and moment-based limiter
\cite{BDF1994ANM, CS1989MC, QS2005SISC, ZS2013JCP}, etc.
Another is to add an artificial diffusion term in the weak formulation, while the artificial
diffusion coefficient should be chosen adequately, see e.g. \cite{Har2006IJNMF,
HM2014NM}. Recently, we proposed a different approach in controlling the spurious
oscillations for computing the hyperbolic conservation laws in
\cite{LLS2021SINUM,LLS2021SISC} and shallow water equations in \cite{LLTX2021JCP}.
The key ingredient is to add a numerical damping term in the existing
DG scheme, and the added term is a high order term if the solution stays smooth and
takes effect whenever the solution is non-smooth.
Fortunately, this approach inherits many good properties such as conservation,
$L^2$-boundedness, optimal error estimates and superconvergence results
from the conventional DG scheme, which makes it quite attractive.
Besides, this approach is so local that it is efficient and friendly for parallel computation.

We proceed to extend this approach to the
convection-diffusion problems \eqref{eqn_cdp} with possibly degenerate diffusion terms.
We adopt the LDG scheme in \cite{Yao2019NA}, in which it considered the
generalized alternating numerical fluxes, which are more general and complex in the
numerical analysis.
Similar to \cite{Yao2019NA}, we can also obtain the $L^2$-boundedness and optimal
error estimates, while the added damping term can be estimated separately.
It should be noted that the optimal error estimates are based on the so-called
{\it generalized Gauss-Radau} (GGR) projection technique developed in \cite{MSW2016MC}.
In the numerical simulation, we test several commonly used equations such as
the porous medium equations, Buckley-Leverett
equations, as well as other degenerate parabolic problems. The numerical results show
that our scheme not only possesses the high order accuracy for the smooth
solutions but also can compress the spurious oscillations effectively. This also verifies
the theoretical results and demonstrates the good performance of the proposed algorithm.

The paper is organized as follows. In Section \ref{sec_1d}, we consider the one-dimensional
degenerate parabolic equation (DPE) and propose an oscillation free
local discontinuous Galerkin (OFLDG) scheme. The theoretical analysis
on $L^2$-boundedness and optimal
error estimates are derived in the semi-discrete framework.
In Section \ref{sec_2d}, we extend the 1D case to multidimensional
problems and obtain similar theoretical results. In Section \ref{sec_numeric}, we
conduct some numerical experiments by computing different kinds of DPEs, including porous
medium equations, Buckley-Leverett equations in both one and two dimensions.
Concluding remarks are given in Section \ref{sec_sum}.

Throughout this paper, we adopt the standard notations in Sobolev
space. $W^{m,p}(D)$ on the subdomain $D\subset \Omega$ is equipped
with the norm $\|\cdot\|_{W^{m,p}(D)}$.
If $p=2$, $W^{m,2}(D)=H^{m}(D)$,
$\|\cdot\|_{W^{m,2}(D)}=\|\cdot\|_{H^m(D)}$. We use $\|\cdot\|_{D}$ to
denote the $L^{2}$ norm in $D$, if $D=\Omega$ then we omit the subscript $D$.
For all positive integer $N$, we define $Z_N=\{1,\cdots, N\}$.

\section{The OFLDG scheme in one dimension}
\label{sec_1d}

In this section, we present the OFLDG scheme for the following one-dimensional nonlinear
degenerate parabolic equations,
\begin{align}\label{1d_model}
u_t + \big(f(u) - a(u)u_x \big)_x = 0, \quad x\in \Omega=[a,b],~~t\in (0,T],
\end{align}
with the initial condition $u(x,0)=u_0(x)$, $x\in \Omega$ and periodic or compactly supported
boundary conditions. Here,  $f(u)$ is the flux function and $a(u)\geq 0$ is the viscous coefficient.

\subsection{Basic notations}
\label{sec_notation_1d}

Firstly, we give some notations that will be used later. Let $I_h$ be a partition of the domain
$\Omega$, defined as follows:
\begin{align*}
a=x_{\frac{1}{2}}<x_{\frac{3}{2}}<\cdots <x_{N+\frac{1}{2}}=b.
\end{align*}
For $j\in Z_N$, we denote $I_j = \big( x_{\jl},x_{\jr} \big)$, $x_j=\frac{1}{2} \big( x_{\jl}+x_{\jr} \big)$, $h_j=x_{\jr}-x_{\jl}$. Furthermore, we assume that the mesh is quasi-uniform,
i.e. there exists a constant $\gamma>0$ such that
\begin{align}
0<\frac{h}{\rho}<\gamma, \quad \text{ where } ~h=\max_{j}=h_j,\quad \rho=\min_{j}{h_j}.
\end{align}
We define the finite element space as follows,
\begin{align} \label{fes_space}
V_h^{k}=\{v_h\in L^2(\Omega): v_h\big|_{I_j}\in \mathcal{P}^k(I_j),~~j\in Z_N\},
\end{align}
where $\mathcal{P}^k(I_j)$ is the polynomial of degree at most $k$ in $I_j$. We denote the right
and left limits of $v_h$ at $x_{\jr}$ as $(v_h)_{\jr}^{+}$ and $(v_h)_{\jr}^{-}$, respectively.
The jump and average of $v_h$ at $x_{\jr}$ are denoted as:
\begin{align*}
\jump{v_h}_{\jr}=(v_h)_{\jr}^{+} - (v_h)_{\jr}^{-},\quad
\mean{v_h}_{\jr}=\frac{1}{2}\Big((v_h)_{\jr}^{+} + (v_h)_{\jr}^{-}\Big).
\end{align*}
We also define $\displaystyle \mean{v_h}^{\theta}_{\jr}=\theta (v_h)_{\jr}^{-}+(1-\theta) (v_h)_{\jr}^{+}$,
for an arbitrary parameter $\theta$.  We use $\|\cdot\|_{\Gamma_h}$ to denote the semi-norm
on the boundary, defined as follows:
\begin{align*}
\|v_h\|_{\Gamma_h}^2 = \sum_{j}\Big( \big( (v_h)_{\jr}^{+} \big)^2 + \big( (v_h)_{\jr}^{-} \big)^2\Big) \, .
\end{align*}

\subsection{The OFLDG scheme}
\label{sec_ofldg_1d}

To derive the OFLDG method for \eqref{1d_model}, we introduce a new auxiliary variable $q=b(u)u_x$,
with $b(u)=\sqrt{a(u)}$. Then, the resulting system is of the form
\begin{align}
u_t +\big(f(u)-b(u)q\big)_x&= 0, \quad x \in \Omega, \quad t\in (0,T],\label{eq:eqn1.1}\\
q-g(u)_x&=0, \quad  x\in \Omega, \label{eq:eqn1.2}
\end{align}
where $g(u)=\displaystyle\int^{u}b(u)du$ is the diffusion flux for the auxiliary variable $q$.
We define the unknown $\bm{w}=(u,q)^T$ and the flux function
\[\bm{h}(\bm{w})=\big(h_u(\bm{w}),h_q(\bm{w})\big)^T=\big(f(u)-b(u)q,\,-g(u)\big)^T.\]
The semi-discrete OFLDG scheme is defined as follows:
seek $\bm{w}_h=(u_h,q_h)^T\in [V_h^k]^2$ such that
for any test functions $v_h$, $r_h \in V_h^k$ and $j\in Z_N$, we have
\begin{align}
((u_h)_t, v_h)_j&= H_{j}\big(h_u(\bm{w}_h),v_h\big)+D_{j}(u_h,v_h),\label{scheme1.1}\\
(q_h, r_h)_j&=G_{j}\big(h_q(\bm{w}_h), r_h\big),\label{scheme1.2}
\end{align}
where, $H_j(\cdot,\cdot)$, $G_j(\cdot,\cdot)$ and $D_j(\cdot,\cdot)$ are defined as follows:
\begin{align*}
&H_{j}\big(h_u(\bm{w}_h),v_h\big)=\big(h_u(\bm{w}_h), (v_h)_x\big)_j
-\widehat{h_{u}}(\bm{w}_h)_{\jr}(v_h)_{\jr}^{-}+\widehat{h_{u}}(\bm{w}_h)_{\jl}(v_h)_{\jl}^{+},\\
&G_{j}\big(h_q(\bm{w}_h), r_h\big)=\big(h_q(\bm{w}_h), (r_h)_x\big)_j
- \widehat{h_{q}}(\bm{w}_h)_{\jr}(r_h)_{\jr}^{-}+\widehat{h_{q}}(\bm{w}_h)_{\jl}(r_h)_{\jl}^{+},\\
&D_{j}(u_h,v_h)=- \sum_{\ell=0}^{k}\frac{\sigma_j^\ell(u_h)}{h_j}\int_{I_j}(u_h-P_h^{\ell-1}u_h)v_h \, dx.
\end{align*}
Here, we use the notation $\displaystyle (r, v)_j=\int_{I_j} rv \, dx$, for all $r, v\in L^{2}(I_j)$.
The ``hat'' terms are numerical fluxes, defined as
\begin{align}
&\widehat{h_{u}}(\bm{w}_h)_{\jr}=\hat{f} \big( (u_h)_{\jr}^{-},(u_h)_{\jr}^{+} \big)
- \frac{\jpl g(u_h)\jpr_{\jr}}{\jpl u_h \jpr_{\jr}}\mean{q_h}_{\jr}-\gamma\jpl q_h \jpr_{\jr},\label{1dflux1}\\
&\widehat{h_{q}}(\bm{w}_h)_{\jr}=-\mean{g(u_h)}_{\jr}+\gamma\jpl u_h \jpr_{\jr},\quad
\gamma=\Big(\theta-\frac{1}{2} \Big) \frac{\jump{g(u_h)}_{\jr}}{\jpl u_h \jpr_{\jr}}, \quad \theta \in \mathbb{R}, \label{1dflux2}
\end{align}
where, $\hat{f}\big( (u_h)_{\jr}^{-},(u_h)_{\jr}^{+} \big)$ is a monotone flux for $f(u)$, such as the Lax-Friedrichs flux \cite{CS1989MC}.
$D_j(u_h, v_h)$ in \eqref{scheme1.1} is the damping term to control spurious oscillations.
In particular, $P_h^{\ell-1}$ in the damping term is the standard local $L^2$ projection into $V_h^{\ell-1}$,
$\ell=1,\cdots,k$, and we define $P_h^{-1}=P_h^{\,0}$. Parameter $\sigma_j^\ell(u_h)$ is the damping coefficient
taken as following:
\begin{align}
\sigma_j^\ell(u_h)=\frac{2(2\ell+1)h^\ell}{(2k-1)\ell!}\Big(\jump{\partial_x^\ell u_h}_{\jr}^2
 +\jump{\partial_x^\ell u_h}_{\jl}^2\Big)^{\frac{1}{2}},\quad 0 \leq \ell \leq k, \, k \geq 1 .
\end{align}
Next, we will present the $L^2$ stability and optimal error estimates results for the OFLDG scheme
\eqref{scheme1.1}-\eqref{scheme1.2}.

\begin{thm}
For periodic or compactly supported boundary conditions, the solution  $\bm{w}_h=(u_h,q_h)^T$
of the semi-discrete OFLDG scheme (\ref{scheme1.1})-(\ref{scheme1.2}) satisfies the following $L^2$ stability, i.e
\begin{align}
\frac{1}{2}\frac{d}{dt}\|u_h\|^2+\|q_h\|^2\leq 0.\label{1dstability}
\end{align}
\end{thm}
\begin{proof}
We take $v_h=u_h$, $r_h=q_h$ in (\ref{scheme1.1}) and (\ref{scheme1.2}) respectively. After summing it over $j$, we have
\begin{align*}
\frac{1}{2}\frac{d}{dt}\|u_h\|^2+\|q_h\|^2&=\sum_j\Big(H_{j}\big(h_u(\bm{w}_h),u_h\big)
+ G_{j}\big(h_q(\bm{w}_h), q_h\big)\Big)+\sum_jD_{j}(u_h,u_h)\\
&=-\sum_j\Theta_{\jr}+\sum_jD_{j}(u_h,u_h).
\end{align*}
Then, (\ref{1dstability}) follows from
\begin{align*}
\Theta_{\jr} & =\int_{(u_h)_{\jr}^{-}}^{(u_h)_{\jr}^{+}}\Big(f(y) - \hat{f} \big( (u_h)_{\jr}^{-},(u_h)_{\jr}^{+}\big) \Big) \, dy\geq 0, \\
 D_{j}(u_h,u_h) & =- \sum_{\ell=0}^{k}\frac{\sigma_j^\ell(u_h)}{h_j}\int_{I_j}(u_h-P_h^{\ell-1}u_h)u_h \, dx\\
& =- \sum_{\ell=0}^{k}\frac{\sigma_j^\ell(u_h)}{h_j}\int_{I_j}(u_h-P_h^{\ell-1}u_h)^2 \, dx\leq 0.
\end{align*}
\end{proof}

\subsection{Error estimates}
\label{sec_errest_1d}
In this subsection, we give the optimal error estimates of the OFLDG scheme for smooth solutions
of  (\ref{1d_model}) with periodic boundary conditions and smooth initial conditions.
We follow the similar approach in \cite{Yao2019NA}. Since the additional damping term is used in the
OFLDG scheme to control the spurious oscillations, we need to prove that the damping term would
not destroy the accuracy. Due to the nonlinear nature of the flux function $\bm{h}(\bm{w})$, we treat it
by Taylor expansion as in \cite{Yao2019NA,Zhang2004SINUM}.
Therefore, we need {\em a priori assumption} that for  sufficiently small $h$, there holds
\begin{align}\label{prioriassump}
\max_{t\in[0,T]}\|u-u_h\|_{L^\infty(\Omega)}\leq C h.
\end{align}
This assumption is frequently used in the analysis of nonlinear problems. For the linear flux functions,
i.e. $f(u)=cu$, the assumption is not necessary. In fact, this assumption can be justified for $k\geq 1$,
see \cite{Meng2012SINUM, Zhang2004SINUM}.  To utilize the Taylor expansion, we need to ensure that
the $f(u)$ and $b(u)$ and their derivatives are bounded. Hence, we assume $f(u)$ and $b(u)\in C^2$.

\subsubsection{One-dimensional projection}
\label{sec_projection_1d}

First of all, we present the projection that will be used in the error estimates.
For a given vector function $\bm{v}=(v_1,\,v_2)^T\in [H^1(\Omega)]^2$, we define the projection $\Pi\bm{v}$:
\begin{align*}
\Pi\bm{v}=(\mathbb{G}_\theta\,v_1,\,\widetilde{\mathbb{G}}_{\theta} \,v_2)^T\in [V_h^{k}]^2\, ,
\end{align*}
where  $\mathbb{G}_\theta\,v_1$ is the generalized Gauss-Radau (GGR) projection of $v_1$ satisfying
\begin{align}
\int_{I_j} (\mathbb{G}_\theta v_1)v_h\,dx&=\int_{I_j} v_1v_h\,dx, \quad \forall \, v_h\in \mathcal{P}^{k-1}(I_j),~j\in Z_N,\label{1dprojection1}\\
\mean{\mathbb{G}_\theta v_1}^{\theta}_{\jr}&=\mean{v_1}^{\theta}_{\jr},\quad\quad \forall \, j\in Z_N.\label{1dprojection2}
\end{align}
$\widetilde{\mathbb{G}}_{\theta} \,v_2$ is defined as follows:
\begin{align}
\int_{I_j} (\widetilde{\mathbb{G}}_{\theta}v_2)v_h\,dx&=\int_{I_j} v_2 v_h\,dx,
\quad\quad \forall \, v_h\in  \mathcal{P}^{k-1}(I_j),~j\in Z_N,\label{1dprojection3}\\
\mean{\widetilde{\mathbb{G}}_{\theta} v_2}^{\tilde{\theta}}_{\jr}&=\mean{v_2}^{\tilde{\theta}}_{\jr}
+ \Big(\theta-\frac{1}{2} \Big) \big(b(v_1)_x\jump{v_1- \mathbb{G}_\theta v_1} \big)_{\jr},
\quad \forall \, j\in Z_N.\label{1dprojection4}
\end{align}
Throughout this paper, we denote $\tilde{\theta}=1-\theta$ for convenience. For the projection $\Pi$,
there exists the following approximation results which were shown in \cite[Lemma 3.1]{Yao2019NA}:
\begin{lem}
If $\bm{v}=(v_1,v_2)^T\in [H^{s+1}(\Omega)]^2$, $s\geq 0$, $\theta> 1/2$. The projection
$\Pi: [H^{1}(\Omega)]^2\rightarrow[V_h^{k}]^2$ is well defined by (\ref{1dprojection1})-(\ref{1dprojection4}). Moreover, there holds the approximation property
\begin{align}\label{pro1d_est}
\|\eta_{\bm{v}}^i\|+h^{\frac{1}{2}}\|\eta_{\bm{v}}^i\|_{\Gamma_h}
\leq Ch^{\min(k,s)+1}\Big(\|v_1\|_{H^{s+1}(\Omega)}+\|v_2\|_{H^{s+1}(\Omega)}\Big),
\end{align}
where $i=1, 2$, $\eta_{\bm{v}}^1=v_1-\mathbb{G}_\theta\,v_1, \eta_{\bm{v}}^2
= v_2-\widetilde{\mathbb{G}}_{\theta}v_2$, and $C$ is a positive constant independent of $h$.
\end{lem}

\subsubsection{An optimal error estimate}
\label{sec_optimal_errest_1d}

In this section, we present an optimal error estimate for the semi-discrete OFLDG method (\ref{scheme1.1})-(\ref{scheme1.2}). To this end,
we assume $f^\prime(u)\geq 0$ and adopt the upwind-biased numerical flux for $f(u)$.
\begin{thm}\label{errorest1d}
Let $\bm{w}=(u,q)^{T}$ be the exact solution of the equation \eqref{eq:eqn1.1}-\eqref{eq:eqn1.2}.
Suppose $u(x,t)\in L^{\infty}\big((0,T); H^{k+1}(\Omega)\big)$, $u_t(x,t)\in L^{2}\big((0,T); H^{k+1}(\Omega)\big)$,
$b(u),\,f(u)\in C^2$ and $f^\prime(u)\geq 0$.
Let $\bm{w}_h=(u_h,q_h)^T$ be the solution of the semi-discrete OFLDG scheme (\ref{scheme1.1})-(\ref{scheme1.2})
with the numerical fluxes (\ref{1dflux1})-(\ref{1dflux2}) and
\begin{align}\label{upwind_biased_flux}
\hat{f}\big( (u_h)_{\jr}^{-},(u_h)_{\jr}^{+} \big) = \theta f\big( (u_h)_{\jr}^{-} \big)
+ (1-\theta)f\big( (u_h)_{\jr}^{+} \big), \quad \theta>\frac{1}{2}.
\end{align}
The initial approximation is taken as $u_h(\cdot,0)=P_h^k u(\cdot,0)$, $P_h^k$ is the standard local $L^2$ projection.
Then we have the following optimal error estimate
\begin{align}\label{error_est_1d}
\|u(T)-u_h(T)\|\leq Ch^{k+1}, \quad k\geq 1,
\end{align}
where $C>0$ is a constant depending on $u$ and its derivatives but independent of $h$.
\end{thm}

\begin{proof}
Firstly, we rewrite the error $e_{\bm{w}}=\bm{w}-\bm{w_h}=(u-u_h,p-p_h)^T$ in two parts:
$$e_u=u-u_h=\eta_u-\xi_u,~\eta_u=u-\mathbb{G}_\theta\,u,~\xi_u=u_h-\mathbb{G}_\theta\,u;$$
$$e_q=q-q_h=\eta_q-\xi_q,~~\eta_q=q-\widetilde{\mathbb{G}}_{\theta} \,q,~~~\xi_q
= q_h-\widetilde{\mathbb{G}}_{\theta} \,q.$$
Since the exact solution $\bm{w}=(u,q)^T$ also satisfies the OFLDG scheme
\eqref{scheme1.1}-\eqref{scheme1.2}, we have the following error equations: $\forall \, v_h, r_h\in V_{h}^k$,
\begin{align}
((e_u)_t, v_h)_j&= H_{j}\big(h_u(\bm{w})-h_u(\bm{w}_h),v_h\big)-D_{j}(u_h,v_h),\label{erroreq1_1d}\\
(e_q, r_h)_j&=G_{j}\big(h_q(\bm{w})-h_q(\bm{w}_h), r_h\big).\label{erroreq2_1d}
\end{align}
Taking $v_h=\xi_u$, $r_h=\xi_q$ and adding up \eqref{erroreq1_1d}-\eqref{erroreq2_1d}, we obtain
\begin{align*}
((\xi_u)_t, \xi_u)_j+(\xi_q, \xi_q)_j=&((\eta_u)_t, \xi_u)_j+(\eta_q, \xi_q)_j-H_{j}\big(h_u(\bm{w})-h_u(\bm{w}_h),\xi_u\big)\\
&-G_{j}\big(h_q(\bm{w})-h_q(\bm{w}_h), \xi_q\big)+D_{j}(u_h,\xi_u).
\end{align*}
Summing over $j$, we have
\begin{equation}
\begin{split}
\frac{1}{2}\frac{d}{dt}\|\xi_u\|^2+\|\xi_q\|^2
=&\sum_{j=1}^N((\eta_u)_t, \xi_u)_j+\sum_{j=1}^N(\eta_q, \xi_q)_j+\sum_{j=1}^N D_j(u_h,\xi_u)\\
&\!-\!\sum_{j=1}^{N}\Big(H_{j}\big(h_u(\bm{w})\!-\!h_u(\bm{w}_h),\xi_u\big)+G_{j}\big(h_q(\bm{w})-h_q(\bm{w}_h), \xi_q\big) \Big).\label{eq1d}
\end{split}
\end{equation}
Now we proceed to estimate the terms in the right hand side of \eqref{eq1d}. First we have
\begin{align}
\sum_{j=1}^N((\eta_u)_t, \xi_u)_j&\leq Ch^{k+1}\|\xi_u\|\leq \frac{1}{4}\|\xi_u\|^2+Ch^{2k+2},\label{est1.1}\\
\sum_{j=1}^N(\eta_q, \xi_q)_j&\leq Ch^{k+1}\|\xi_q\|\leq \frac{1}{4}\|\xi_q\|^2+Ch^{2k+2}.\label{est1.2}
\end{align}
%
With the help of the a priori assumption \eqref{prioriassump}, we could get the estimates
for the last term in \eqref{eq1d} as in \cite[Lemma 3.2, Lemma 3.3]{Yao2019NA},
\begin{align}
-\sum_{j=1}^{N}\Big(H_{j}\big(h_u(\bm{w})-h_u(\bm{w}_h),\xi_u\big)
+ G_{j}\big(h_q(\bm{w})-h_q(\bm{w}_h), \xi_q\big) \Big)
\leq \frac{1}{4}\|\xi_q\|^2+C\|\xi_u\|^2 + Ch^{2k+2}.\label{est1.3}
\end{align}
For the damping term $D_j(u_h,\xi_u)$, we have
\begin{align*}
\sum_{j=1}^N D_j(u_h,\xi_u)&=-\sum_{j=1}^{N}\sum_{\ell=0}^{k}
\frac{\sigma_j^\ell (u_h)}{h_j}\int_{I_j}(u_h-P_h^{\ell-1}u_h)\xi_u \, dx  \\
&=-\sum_{j=1}^{N}\sum_{\ell=0}^{k}\frac{\sigma_j^\ell(u_h)}{h_j}
\int_{I_j}\Big(\xi_u-P_h^{\ell-1}\xi_u\Big)^2+\Big(\mathbb{G}_\theta\,u-P_h^{\ell-1}(\mathbb{G}_\theta\,u)\Big) \xi_u\, dx\\
&\leq -\sum_{j=1}^{N}\sum_{\ell=0}^{k}\frac{\sigma_j^\ell(u_h)}{h_j}
\int_{I_j}\Big(\mathbb{G}_\theta\,u-P_h^{\ell-1}(\mathbb{G}_\theta\,u)\Big) \xi_u\, dx\\
&\leq\sum_{j=1}^{N}\sum_{\ell=0}^{k}\frac{\sigma_j^\ell(u_h)}{h_j}
\|\mathbb{G}_\theta\,u-P_h^{\ell-1}(\mathbb{G}_\theta\,u)\|_{L^2(I_j)} \|\xi_u\|_{L^2(I_j)}.
\end{align*}
Thanks to the properties of projections $\mathbb{G}_\theta$ and $P_h^{\ell-1}$, we have
\begin{align*}
\|\mathbb{G}_\theta\,u-P_h^{\ell-1}(\mathbb{G}_\theta\,u)\|_{L^2({I_j})}
\leq &\|\mathbb{G}_\theta\,u-u\|_{L^2({I_j})}+\|u-P_h^{\ell-1}u\|_{L^2({I_j})}
+ \| P_h^{\ell-1}(\mathbb{G}_\theta\,u-u)\|_{L^2(I_j)}\\
\leq \,&2\|\mathbb{G}_\theta\,u-u\|_{L^2({I_j})}+\|u-P_h^{\ell-1}u\|_{L^2({I_j})}\\
\leq\,&Ch^{k+1}\|u\|_{H^{k+1}(\Omega)}+Ch^{\max(\ell,1)+\frac 12}\|u\|_{W^{\max(\ell,1),\infty}(\Omega)}\\
\leq \,& Ch^{k+1}\|u\|_{H^{k+1}(\Omega)}+Ch^{\max(\ell,1)+\frac{1}{2}}\|u\|_{H^{\max(\ell,1)+1}(\Omega)}\\
\leq \,& Ch^{\max(\ell,1)+\frac{1}{2}}\|u\|_{H^{k+1}(\Omega)}.
\end{align*}
For the coefficient $\sigma_j^\ell(u_h)$, we have
\begin{align*}
\sigma_j^\ell(u_h)^2&=\frac{4(2\ell+1)^2h^{2\ell}}{(2k-1)^2(\ell!)^2}
\Big(\jump{\partial_x^\ell(u_h-u)}_{\jl}^2+\jump{\partial_x^\ell(u_h-u)}_{\jr}^2\Big)\\
&\leq C h^{2\ell}\Big(\jump{\partial_x^\ell \xi_u}^2_{\jl}+\jump{\partial_x^\ell \xi_u}^2_{\jr}\Big)
+ Ch^{2\ell}\Big(\jump{\partial_x^\ell \eta_u}^2_{\jl}+\jump{\partial_x^\ell \eta_u}^2_{\jr}\Big).
\end{align*}
Thus, we have
\begin{align*}
\sum_{j=1}^N D_j(u_h,\xi_u)
\leq& \sum_{j=1}^{N}\sum_{\ell=0}^{k}Ch^{\ell+\max(\ell,1)-\frac{1}{2}}
\Big(\jump{\partial_x^\ell \xi_u}^2_{\jl}+\jump{\partial_x^\ell \xi_u}^2_{\jr}\Big)^{\frac 12}
\|\xi_u\|_{L^2(I_j)}\\
&+\sum_{j=1}^{N}\sum_{\ell=0}^{k}Ch^{\ell+\max(\ell,1)-\frac{1}{2}}
\Big(\jump{\partial_x^\ell \eta_u}^2_{\jl}+\jump{\partial_x^\ell \eta_u}^2_{\jr}\Big)^{\frac 12}\|\xi_u\|_{L^2(I_j)}\\
\leq&C\left(\Big(\sum_{j=1}^{N}\sum_{\ell=0}^{k}h^{2\ell+1}\jump{\partial_x^\ell \xi_u}^2_{\jl}\Big)^{\frac 12}
+ \Big(\sum_{j=1}^{N}\sum_{\ell=0}^{k}h^{2\ell+1}\jump{\partial_x^\ell \eta_u}^2_{\jl}\Big)^{\frac 12}\right)\|\xi_u\| \\
\leq&C \|\xi_u\|^2+C\|\eta_u\| \|\xi_u\|.
\end{align*}
By the Cauchy-Schwarz inequality and \eqref{pro1d_est}, we have
\begin{align}
\sum_{j=1}^N D_j(u_h,\xi_u) \leq C\|\xi_u\|^2+Ch^{2k+2}. \label{est1.5}
\end{align}
Therefore, combining equations (\ref{est1.1})-(\ref{est1.5}), we have
\begin{align*}
\frac{1}{2}\frac{d}{dt}\|\xi_u\|^2+\|\xi_q\|^2\leq Ch^{2k+2}+C\|\xi_u\|^2+\frac{1}{2}\|\xi_q\|^2.
\end{align*}
With the Gr\"onwall's inequality and initial discretization, we can obtain
\begin{align}
\|\xi_u\|\leq C h^{k+1}.
\end{align}
Finally, combining with the triangle inequality, we obtain the optimal error estimate
\eqref{error_est_1d}.
\end{proof}

\begin{rem}
Note that the upwind biased flux \eqref{upwind_biased_flux} is chosen only to obtain the optimal error estimates. One can also obtain the $(k+\frac 12)$-th order error estimate for the monotone numerical flux by the analogous arguments in \cite{Xu2007CMAME}.
\end{rem}

\section{The OFLDG scheme in multidimensions}
\label{sec_2d}

In this section, we extend the OFLDG method to the multidimensional case.
For simplicity, we only consider the two-dimensional space, and the higher dimensional cases can be
obtained directly by the same line as the two-dimensional one.
We now consider the two-dimensional nonlinear degenerate parabolic equations:
\begin{align}
u_t+\big(f_1(u)-a_{11}(u)u_x-a_{12}(u)u_y\big)_x+\big(f_2(u)-a_{21}(u)u_x-a_{22}(u)u_y\big)_y=0,\label{2d_model}
\end{align}
with the periodic boundary conditions or compactly supported boundary conditions.  $(x, y) \in \Omega = [a_x, b_x]\times [a_y, b_y]$, $t\in (0,T]$, and
$f_1(u)$, $f_2(u)$ are convective flux functions. The diffusion tensor $\bm{a}(u)$
is positive semidefinite and given as
\begin{align*}
\bm{a}(u) =
\begin{pmatrix}
a_{11}(u) & a_{12}(u) \\
a_{21}(u) & a_{22}(u)
\end{pmatrix} .
\end{align*}
Without loss of generality, we take $a_{11}(u)=a_1(u)\geq 0$, $a_{22}(u)=a_2(u)\geq 0$
and $a_{12}(u)=a_{21}(u)=0.$

\subsection{Basic notations}
\label{sec_notation_2d}

Firstly, we assume that a shape regular tessellation of $\Omega$ is given as $\Omega_h$, with rectangular elements
$$K_{i, j}=I_i\times J_j = \big[ x_{\il},x_{\ir} \big]\times \big[ y_{\jl},y_{\jr} \big],~~i\in Z_{N_x},~~j\in Z_{N_y}.$$
The union of all element boundaries in $\Omega_h$ is denoted as $\Gamma_h$.
We define the finite element space with the partition $\Omega_h$,
\begin{align}
W_h^{k}=\{v_h\in L^2(\Omega): v_h\big|_{K_{i, j}}\in \mathcal{Q}^k(K_{i, j}),~~i\in Z_{N_x}, j\in Z_{N_y}\},
\end{align}
where $\mathcal{Q}^k(K_{i, j}) = \mathcal{P}^{k}(I_i)\otimes  \mathcal{P}^{k}(J_j)$ is the tensor
product of two polynomial spaces in which the polynomial degree is at most $k$ for each variable.
Now we define
$$
h_i^x=x_{\ir}-x_{\il},~h_j^y=y_{\jr}-y_{\jl},~h_{K_{i, j}} = \max\{h_i^x, h_j^y\},~  h = \max_{i, j}\{h_{K_{i, j}}\}.
$$
For $i\in Z_{N_x}, j\in Z_{N_y}$, we denote $(v_h)_{\ir, y}^{\pm} = v_h(x_{\ir}^{\pm},y)$, $(v_h)_{x,\jr}^{\pm}
= v_h(x,y_{\jr}^{\pm})$, $(v_h)_{\ir,\jr}^{\pm,\pm}=v_h(x_{\jr}^{\pm},y_{\jr}^{\pm})$.
Then, we define the average and jump of $v_h$ at  $( x_{\ir}, y)$ and $( x, y_{\jr} )$ as follows,
$$\mean{v_h}_{\ir,y}=\frac{1}{2}\big(v_h(x_{\ir}^+,y)+v_h(x_{\ir}^-,y)\big),
\quad\jump{v_h}_{\ir,y}=v_h(x_{\ir}^+,y)-v_h(x_{\ir}^-,y).$$
$$\mean{v_h}_{x,\jr}=\frac{1}{2}\big(v_h(x,y_{\jr}^+)+v_h(x,y_{\jr}^-)\big),
\quad\jump{v_h}_{x,\jr}=v_h(x,y_{\jr}^+)-v_h(x,y_{\jr}^-).$$
The semi-norm on element boundaries in two-dimensional space is defined as follows
\begin{align*}
\|v_h\|_{\Gamma_h}^2 = \sum_{i,j}\int_{I_i} \big((v_h)^+_{x,\jl}\big)^2 + \big((v_h)^-_{x,\jr}\big)^2\, dx
+ \sum_{i,j}\int_{J_j} \big((v_h)^+_{\il,y}\big)^2+\big((v_h)^-_{\ir,y}\big)^2\, dy .
\end{align*}

\subsection{The OFLDG scheme}
\label{sec_ofldg_2d}

In this section, we present the OFLDG scheme for the two-dimensional nonlinear parabolic equation
\eqref{2d_model}. Similar to the one-dimensional case, we introduce auxiliary variables
$q_1=b_1(u)u_x$, $q_2=b_2(u)u_y$, with $b_1(u)=\sqrt{a_1(u)}$ and $b_2(u)=\sqrt{a_2(u)}$
to rewrite \eqref{2d_model} into a first order system,
\begin{align}
&u_t+\big(f_1(u)-b_1(u)q_1\big)_x+\big(f_2(u)-b_2(u)q_2\big)_y=0,\label{eq:eqn2_1}\\
&q_1-g_1(u)_x=0, \label{eq:eqn2_2}\\
&q_2-g_2(u)_y=0, \label{eq:eqn2_3}
\end{align}
where $\displaystyle g_1(u)=\int^u b_1(u)\,du$, $\displaystyle g_2(u)=\int^u b_2(u)\,du$.
We define the unknown variable $\bm{w}=(u, q_1, q_2)^{T}$ and the flux function
\begin{align*}
\bm{h}(\bm{w})=&\big(h_u^1(\bm{w}), h_u^2(\bm{w}), h_q^1(\bm{w}), h_q^2(\bm{w})\big)^T\\
=& \big(f_1(u)-b_1(u)q_1, \,f_2(u)-b_2(u)q_2, \,-g_1(u), \,-g_2(u)\big)^T.
\end{align*}
The semi-discrete OFLDG scheme is defined as follows:
seek $\bm{w}_h=(u_h,q_{1h},q_{2h})^{T} \in [W_h^k]^3$, such that for all test functions
$v_h, r_h, p_h \in W_h^k$ and $i\in Z_{N_x}, j\in Z_{N_y}$, we have
\begin{align}
\int_{K_{i, j}} (u_h)_t v_hdxdy=&H_{ij}^1\big( h_u^1(\bm{w}_h), v_h\big)+H_{ij}^2\big( h_u^2(\bm{w}_h), v_h\big)+D_{ij}(u_h,v_h),\label{scheme2.1}\\
\int_{K_{i, j}} q_{1h}\, r_hdxdy=&G_{ij}^1\big( h_q^1(\bm{w}_h), r_h\big),\label{scheme2.2}\\
\int_{K_{i, j}} q_{2h}\, p_hdxdy=&G_{ij}^2\big( h_q^2(\bm{w}_h), p_h\big),\label{scheme2.3}
\end{align}
where $H_{ij}^1(\cdot, \cdot),~H_{ij}^2(\cdot, \cdot),~G_{ij}^1(\cdot, \cdot),~G_{ij}^2(\cdot, \cdot)$ and $D_{ij}(\cdot, \cdot)$ are defined as follows:
\begin{align*}
H_{ij}^1\big( h_u^1(\bm{w}_h), v_h\big)=&\int_{K_{i, j}} h_u^1(\bm{w}_h) (v_h)_x \, dxdy
- \int_{J_j}\big(\widehat{ h_u^1}(\bm{w}_h) v_h^-\big)_{\ir,y}-\big(\widehat{ h_u^1}(\bm{w}_h) v_h^+\big)_{\il,y} \, dy,\\
H_{ij}^2\big( h_u^2(\bm{w}_h), v_h\big)=&\int_{K_{i, j}} h_u^2(\bm{w}_h) (v_h)_y \, dxdy
- \int_{I_i}\big(\widehat{ h_u^2}(\bm{w}_h) v_h^-\big)_{x,\jr}-\big(\widehat{ h_u^2}(\bm{w}_h) v_h^+\big)_{x,\jl} \, dx,\\
G_{ij}^1\big( h_q^1(\bm{w}_h), r_h\big)=&\int_{K_{i, j}} h_q^1(\bm{w}_h) (r_h)_x \, dxdy
- \int_{J_j}\big(\widehat{ h_q^1}(\bm{w}_h) r_h^-\big)_{\ir,y}-\big(\widehat{ h_q^1}(\bm{w}_h) r_h^+\big)_{\il,y} \, dy,\\
G_{ij}^2\big( h_q^2(\bm{w}_h), p_h\big)=&\int_{K_{i, j}} h_q^2(\bm{w}_h) (p_h)_y\, dxdy
- \int_{I_i}\big(\widehat{ h_q^2}(\bm{w}_h) p_h^-\big)_{x,\jr}-\big(\widehat{ h_q^2}(\bm{w}_h) p_h^+\big)_{x,\jl} \, dx,\\
D_{ij}(u_h,v_h) = &-\sum_{\ell=0}^{k}\frac{\sigma^\ell_{K_{i, j}}(u_h)}{h_{K_{i, j}}}\int_{K_{i, j}}(u_h-P_h^{\ell-1}u_h)v_h \, dxdy.
\end{align*}
The numerical fluxes are taken as follows:
\begin{align}
\widehat{ h_u^1}(\bm{w}_h)_{\ir, y}&\,= \hat{f_1}\big( (u_{h})_{\ir, y}^-, (u_{h})_{\ir, y}^+ \big)
- \frac{\jpl g_1(u_h)\jpr_{\ir,y}}{\jpl u_h \jpr_{\ir,y}}\mean{q_{1h}}_{\ir,y}-\gamma_1\jpl q_{1h} \jpr_{\ir,y}, \label{2dflux1}\\
\widehat{ h_u^2}(\bm{w}_h)_{x,\jr}&= \hat{f_2}\big((u_{h})_{x, \jr}^-, (u_{h})_{x, \jr}^+ \big)
\!-\! \frac{\jpl g_2(u_h)\jpr_{x,\jr}}{\jpl u_h \jpr_{x,\jr}}\mean{q_{2h}}_{x,\jr} - \gamma_2\jpl q_{2h} \jpr_{x,\jr},\label{2dflux2}\\
\widehat{ h_q^1}(\bm{w}_h)_{\ir,y}&=-\mean{g_1(u_h)}_{\ir,y}+\gamma_1\jpl u_h \jpr_{\ir,y},\label{2dflux3}\\
\widehat{ h_q^2}(\bm{w}_h) _{x,\jr}&=-\mean{g_2(u_h)}_{x,\jr}+\gamma_2\jpl u_h \jpr_{x,\jr},\label{2dflux4}
\end{align}
where
\[\gamma_1 = \Big(\theta_1-\frac{1}{2} \Big) \frac{\jpl g_1(u_h)\jpr_{\ir,y}}{\jpl u_h \jpr_{\ir,y}}, \quad
\gamma_2 = \Big( \theta_2-\frac{1}{2} \Big) \frac{\jpl g_2(u_h)\jpr_{x,\jr}}{\jpl u_h \jpr_{x,\jr}},\quad
\theta_1, ~\theta_2 \in \mathbb{R}.\]
The $P_h^{\ell-1}$ in the damping term is the standard local $L^2$ projection into $W_h^{\ell-1}, \ell=1,\cdots,k$,
and we define $P_h^{-1}=P_h^{0}$.
The damping coefficient $\sigma_{K_{i, j}}^\ell(u_h)$ is defined as follows:
\begin{align}
\sigma_{K_{i, j}}^\ell(u_h)=\frac{2(2\ell+1)}{(2k-1)}\frac{h^{\ell}}{\ell!}\sum_{|\bm{\alpha}|
=\ell}\Big(\frac{1}{N_e}\sum_{\bm{v}\in K_{i, j}}\big(\jump{\partial^{\bm{\alpha}}u_h}\Big|_{\bm{v}}\big)^2\Big)^{\frac 12}.
\end{align}
Here we only consider the jump of $u_h$ on the vertex $\bm{v}$ of two adjacent cells which are shared with edge. $N_e$ is number of vertexes of $K_{i,j}$.
For more details, see \cite{LLS2021SINUM}.
For $L^2$-stability of the scheme \eqref{scheme2.1}-\eqref{scheme2.3}, we have the following theorem:
\begin{thm}
We assume that simulation over $K_{i, j}\in \Omega_h$ with the periodic or compactly supported boundary conditions,
then the solution $\bm{w}_h=(u_h,q_{1h},q_{2h})^{T}$ of the semi-discrete OFLDG scheme
(\ref{scheme2.1})-(\ref{scheme2.3}) with the numerical fluxes (\ref{2dflux1})-(\ref{2dflux4}) is stable in the $L^2$ norm, i.e
\begin{align}
\frac{1}{2}\frac{d}{dt}\|u_h\|^2+\|q_{1h}\|^2+\|q_{2h}\|^2\leq 0.
\end{align}
\end{thm}
The proof of this theorem is similar to the one-dimensional case and omit it here.

\subsection{Error estimates}
\label{sec_errest_2d}

In this subsection, we consider the error estimate of the OFLDG scheme  (\ref{scheme2.1})-(\ref{scheme2.3}) with the periodic boundary condition.
Actually, comparing to the LDG method in \cite{Yao2019NA}, the proposed OFLDG scheme has an additional
damping term to control the spurious oscillation.
Hence, we only need to prove that the damping term does not destroy the accuracy.
Due to the nonlinear nature of the fluxes, {\em a priori assumption} \eqref{prioriassump} is needed in our proof,
In fact, this assumption can be justified for $k\geq 2$ in two-dimensional case,
see \cite{Zhang2004SINUM}.
For the linear flux functions, the assumption is not necessary. Similar to one-dimensional case, we
assume $f_i(u)$ and $b_i(u) \in C^2, ~i=1,2$.

\subsubsection{Two-dimensional projections}
\label{sec_projection_2d}
For a given vector function
$\bm{v}=(v_1,v_2,v_3)^T\in H^{2}(\Omega)\times H^{1}(\Omega)\times H^{1}(\Omega)$,
we define the projection $\bm{\Pi}$:
\begin{align}
\bm{\Pi}
 \bm{v} = \big(\mathbb{G}_{\theta_1,\theta_2}v_1,\widetilde{\mathbb{G}}_{\theta_1,\frac{1}{2}} v_2,
 \widetilde{\mathbb{G}}_{\frac{1}{2},\theta_2} v_3 \big)^T\in [W_h^{k}]^3 \, .
\end{align}
$\bullet$ $\mathbb{G}_{\theta_1,\theta_2} v_1$ is the two-dimensional GGR projection of $v_1$,
defined as follows: for all $i\in Z_{N_x}$, $j\in Z_{N_y}$
\begin{align}
\int_{K_{i, j}}(\mathbb{G}_{\theta_1,\theta_2}v_1)\,r_hdxdy&=\int_{K_{i, j}}v_1r_h \, dxdy,
\qquad \forall \, r_h\in  \mathcal{Q}^{k-1}(K_{i, j}),\\
\int_{J_j}\mean{\mathbb{G}_{\theta_1,\theta_2}v_1}_{\ir,y}^{\theta_1,y}\,r_h \, dy
&\, =\int_{J_j}\mean{v_1}_{\ir, y}^{\theta_1,y}r_hdy, \quad \forall \, r_h\in  \mathcal{P}^{k-1}(J_j),\\
\int_{I_i}\mean{\mathbb{G}_{\theta_1,\theta_2}v_1}_{x,\jr}^{x,\theta_2}\,r_h \, dx
&\,=\int_{I_i}\mean{v_1}_{x,\jr}^{x,\theta_2}r_h \, dx, \quad \forall \, r_h\in  \mathcal{P}^{k-1}(I_i),\\
\mean{\mathbb{G}_{\theta_1,\theta_2}v_1}_{\ir,\jr}^{\theta_1,\theta_2}&=\mean{v_1}_{\ir,\jr}^{\theta_1,\theta_2}.
\end{align}
Here and below, we use the following notations:
\begin{align*}
&\mean{v}_{\ir, y}^{\theta_1,y} = \theta_1v_{\ir, y}^-+\widetilde{\theta_1}v_{\ir, y}^+,
\qquad\mean{v}_{x,\jr}^{x,\theta_2} = \theta_2v_{x, \jr}^-+\widetilde{\theta_2}v_{x, \jr}^+,\\
&\mean{v}_{\ir,\jr}^{\theta_1,\theta_2}=\theta_1\theta_2v_{\ir,\jr}^{-,-}+\theta_1\widetilde{\theta_2}v_{\ir,\jr}^{-,+}+\widetilde{\theta_1}\theta_2v_{\ir,\jr}^{+,-}+\widetilde{\theta_1}\widetilde{\theta_2}v_{\ir,\jr}^{+,+}.
\end{align*}
$\bullet$ $\widetilde{\mathbb{G}}_{\theta_1,\frac{1}{2}} v_2$ is defined in the following:
for all $i\in Z_{N_x}$, $j\in Z_{N_y}$
\begin{align}
& \int_{K_{i, j}} (\widetilde{\mathbb{G}}_{\theta_1,\frac{1}{2}}v_2) \, r_h \, dxdy = \int_{K_{i, j}}v_2 r_h \, dxdy,
\quad \forall \, r_h\in  \mathcal{P}^{k-1}(I_i)\otimes  \mathcal{P}^{k}(J_j),\\
&  \begin{aligned}
 \int_{J_j}\mean{\widetilde{\mathbb{G}}_{\theta_1,\frac{1}{2}}v_2}_{\ir,y}^{\tilde{\theta}_1,y}r_h \, dy & \, =
 \Big( \theta_1-\frac{1}{2} \Big) \int_{J_j}\big(b_1(v_1)_x\jump{v_1-\mathbb{G}_{\theta_1,\theta_2}v_1}\big)_{\ir, y}r_h \, dy \\
& \quad + \int_{J_j}\mean{v_2}_{\ir,y}^{\tilde{\theta}_1,y}r_h \, dy,  \quad \forall \, r_h\in  \mathcal{P}^{k}(J_j).
\end{aligned}
\end{align}
$\bullet$ $\widetilde{\mathbb{G}}_{\frac{1}{2},\theta_2} v_3$ is defined in the following:
for all $i\in Z_{N_x}$, $j\in Z_{N_y}$
\begin{align}
& \int_{K_{i, j}} (\widetilde{\mathbb{G}}_{\frac{1}{2},\theta_2}v_3) \, r_h \, dxdy = \int_{K_{i, j}}v_3 r_h \, dxdy,
\quad \forall \, r_h\in  \mathcal{P}^{k}(I_i)\otimes  \mathcal{P}^{k-1}(J_j),\\
& \begin{aligned}
\int_{I_i}\mean{\widetilde{\mathbb{G}}_{\frac{1}{2},\theta_2}v_3}_{x,\jr}^{x,\tilde{\theta}_2}r_h\, dx & \, =
\Big( \theta_2-\frac{1}{2} \Big)\int_{I_i}\big(b_2(v_1)_y\jump{v_1-\mathbb{G}_{\theta_1,\theta_2}v_1}\big)_{x,\jr}r_h \, dx \\
& \quad + \int_{I_i}\mean{v_3}_{x,\jr}^{x,\tilde{\theta}_2}r_h \, dx , \quad\forall \, r_h\in  \mathcal{P}^{k}(I_i).
\end{aligned}
\end{align}
It has been proved that the GGR projection $\mathbb{G}_{\theta_1,\theta_2}$ and projections $\widetilde{\mathbb{G}}_{\theta_1,\frac{1}{2}}$, $\widetilde{\mathbb{G}}_{\frac{1}{2},\theta_2}$ are well defined. Moreover, the approximation property was given in \cite[Lemma 4.1]{Yao2019NA}:
\begin{lem}
Let $\bm{v}=(v_1,v_2,v_3)\in (H^{s+1}(\Omega)\cap H^{2}(\Omega))\times H^{s+1}(\Omega)\times H^{s+1}(\Omega)$, $s\geq0$. For any $\theta_1> \frac{1}{2}$ and $\theta_2>\frac{1}{2}$, the projection $\bm{\Pi} \bm{v} $ is well defined, and
\begin{align}
\|\eta_{\bm{v}}^i\|+h^{\frac{1}{2}}\|\eta_{\bm{v}}^i\|_{\Gamma_h}\leq Ch^{\min(s,k)+1}(\|v_1\|_{H^{s+1}(\Omega)}+\|v_2\|_{H^{s+1}(\Omega)}+\|v_3\|_{H^{s+1}(\Omega)}),
\end{align}
where $i=1,2,3$, $\eta_{\bm{v}}^1=v_1-\mathbb{G}_{\theta_1,\theta_2} v_1$, $\eta_{\bm{v}}^2=v_2-\widetilde{\mathbb{G}}_{\theta_1,\frac{1}{2}} v_2$, $\eta_{\bm{v}}^3=v_3-\widetilde{\mathbb{G}}_{\frac{1}{2},\theta_2} v_3$ and $C$ is a constant independent of $h$.
\end{lem}

\subsubsection{An optimal error estimate}
\label{sec_optimal_errest_2d}

In this subsection, we prove the additional damping term would not destroy the accuracy of the scheme.
Therefore, we only consider a simple case which $f_{i}(u) \geq 0, i=1,2$ in the governing equation \eqref{2d_model}, and also take the upwind-biased numerical flux for $f_1(u)$ and $f_2(u)$ to get the optimal error estimates. Other cases are similar and omit here. Now we present the main result in this subsection as follows.
\begin{thm}
Let $\bm{w}=(u,q_1,q_2)^{T}$ be the exact solution of the equation \eqref{eq:eqn2_1}-\eqref{eq:eqn2_3}.
Suppose $u(x,t)\in L^{\infty}\big((0,T); H^{k+2}(\Omega)\big)$, $u_t(x,t)\in L^{2}\big((0,T); H^{k+1}(\Omega)\big)$,
 $b_i(u),\,f_i(u)\in C^2$ and $f^\prime_i(u)\geq 0, i=1,2$. Let $\bm{w}_h=(u_h,q_{1h},q_{2h})^T$
 be the solution of the semi-discrete OFLDG scheme (\ref{scheme2.1})-(\ref{scheme2.3}) with
 numerical flux (\ref{2dflux1})-(\ref{2dflux4}) and
$$\hat{f_1}\big( (u_{h})_{\ir, y}^-, (u_{h})_{\ir, y}^+ \big) = \theta_1 f_1\big( (u_{h})_{\ir, y}^-\big)
+ (1-\theta_1)f_1\big( (u_{h})_{\ir, y}^+\big), \quad \theta_1>\frac{1}{2},$$
$$\hat{f_2}\big( (u_{h})_{x, \jr}^-, (u_{h})_{x, \jr}^+ \big) = \theta_2 f_2\big( (u_{h})_{x, \jr}^+\big)
+ (1-\theta_2)f_2\big( (u_{h})_{x, \jr}^+\big), \quad \theta_2>\frac{1}{2}.$$
The initial approximation is taken as $u_h(\cdot,\cdot,0)=P_h^k u(\cdot,\cdot,0)$,
$P_h^k$ is the standard local $L^2$ projection. Then we have the optimal error estimate
\begin{align}\label{error_est_2d}
\|u-u_h\|\leq Ch^{k+1}, \quad k\geq 2,
\end{align}
where $C>0$ is a constant independent of $h$.
\end{thm}
\begin{proof}
Similar to one-dimensional case, we also rewrite the error into two parts with the help of the projection $\bm{\Pi} $,
\begin{align*}
&e_u=u-u_h=\eta_u-\xi_u, \quad\quad\eta_u=u-\g\,u, \quad\xi_u=u_h-\g\,u; \\
&e_{q_1}=q_1-q_{1h}=\eta_{q_1}-\xi_{q_1},~~\eta_{q_1}=q_1-\gr \,q_1, ~~\xi_{q_1}=q_{1h}-\gr \,q_1;\\
&e_{q_2}=q_2-q_{2h}=\eta_{q_2}-\xi_{q_2},~~\eta_{q_2}=q_2-\gl \,q_2, ~~\xi_{q_2}=q_{2h}-\gl \,q_2.
\end{align*}
Since the exact solution $\bm{w}=(u,q_1,q_2)^T$ also satisfies the OFLDG scheme(\ref{scheme2.1})-(\ref{scheme2.3}).
Then, for all $v_h, r_h, p_h\in W_h^k$ we have the following error equations
\begin{align}
\int_{K_{i, j}}(e_u)_t v_h \, dxdy= &~ H_{ij}^1\big( h_u^1(\bm{w})-h_u^1(\bm{w}_h), v_h\big)
+ H_{ij}^2\big(h_u^2(\bm{w})-h_u^2(\bm{w}_h), v_h\big)\label{ereq1_2d}\\
&-D_{ij}(u_h,v_h),\nonumber\\
\int_{K_{i, j}} e_{q_1} r_h \, dxdy=&~G_{ij}^1\big(h_q^1(\bm{w})-h_q^1(\bm{w}_h), r_h\big),\label{ereq2_2d}\\
\int_{K_{i, j}} e_{q_2} p_h \, dxdy=&~G_{ij}^2\big( h_q^2(\bm{w})-h_q^2(\bm{w}_h), p_h\big).\label{ereq3_2d}
\end{align}
Taking $v_h=\xi_u$, $r_h = \xi_{q_1}$,  $p_h=\xi_{q_2}$ in \eqref{ereq1_2d}, \eqref{ereq2_2d}
and \eqref{ereq3_2d} respectively, we have
\begin{align*}
\int_{K_{i, j}} (\xi_u)_t \xi_u \, dxdy=&\int_{K_{i, j}} (\eta_u)_t \xi_udxdy-H_{ij}^1\big( h_u^1(\bm{w}) - h_u^1(\bm{w}_h), \xi_u\big)  \\
& - H_{ij}^2\big(h_u^2(\bm{w})- h_u^2(\bm{w}_h), \xi_u\big) + D_{ij}(u_h,\xi_u),\\
\int_{K_{i, j}} \xi_{q_1} \xi_{q_1} \, dxdy=&~\int_{K_{i, j}} \eta_{q_1}\, \xi_{q_1}dxdy-G_{ij}^1\big(h_q^1(\bm{w})-h_q^1(\bm{w}_h), \xi_{q_1}\big),\\
\int_{K_{i, j}} \xi_{q_2}  \xi_{q_2} \, dxdy=&~\int_{K_{i, j}} \eta_{q_2}\, \xi_{q_2}dxdy-G_{ij}^2\big( h_q^2(\bm{w})-h_q^2(\bm{w}_h), \xi_{q_2}\big).
\end{align*}
Summing it over $i, j$, we can obtain
\begin{equation}
\begin{split}\label{eq2d}
&\frac 12\frac{d}{dt}\|\xi_u\|^2+\|\xi_{q_1}\|^2+\|\xi_{q_2}\|^2 \\
 =& ((\eta_u)_t,\xi_u)+(\eta_{q_1}, \xi_{q_1})+(\eta_{q_2}, \xi_{q_2})+D(u_h,\xi_u) \\
& - \Big(H^1\big( h_u^1(\bm{w}) - h_u^1(\bm{w}_h), \xi_u\big) + H^2\big(h_u^2(\bm{w}) - h_u^2(\bm{w}_h), \xi_u\big)\Big) \\
&- \Big(G^1\big(h_q^1(\bm{w}) - h_q^1(\bm{w}_h), \xi_{q_1}\big) +
G^2\big( h_q^2(\bm{w})-h_q^2(\bm{w}_h),  \xi_{q_2}\big)\Big),
\end{split}
\end{equation}
where $(\cdot,\cdot)$ in two dimensions denotes $\displaystyle(r,v)=\sum_{i,j}\int_{K_{i, j}} rv \, dxdy,$
for all $r, v\in L^2(\Omega)$ and
\[D(\cdot,\cdot)=\sum_{i,j}D_{ij}(\cdot,\cdot),\, H^m(\cdot,\cdot)=\sum_{i,j}H_{ij}^m(\cdot,\cdot),\, G^m(\cdot,\cdot)
= \sum_{i,j}G_{ij}^m(\cdot,\cdot),\, m=1,2.\] Firstly, we have
\begin{align}
((\eta_u)_t,\xi_u)+(\eta_{q_1}, \xi_{q_1})+(\eta_{q_2}, \xi_{q_2})\leq \frac14\|\xi_{q_1}\|^2
+\frac14\|\xi_{q_2}\|^2+\frac14\|\xi_u\|^2+Ch^{2k+2}.\label{est1_2d}
\end{align}
With the help of the a priori assumption \eqref{prioriassump}, we could get the following estimates
 for the last two terms in \eqref{eq2d} as in \cite[Lemma 4.3, Lemma 4.4]{Yao2019NA}. 
\begin{align}
&-\Big(H^1\big( h_u^1(\bm{w})-h_u^1(\bm{w}_h), \xi_u\big)
+ H^2\big(h_u^2(\bm{w})-h_u^2(\bm{w}_h), \xi_u\big)\Big)\leq C\|\xi_u\|^2+Ch^{2k+2},\label{est2_2d} \\
& \begin{aligned} \label{est3_2d}
&-\Big(G^1\big(h_q^1(\bm{w}) - h_q^1(\bm{w}_h), \xi_{q_1}\big) +
G^2\big( h_q^2(\bm{w})-h_q^2(\bm{w}_h),  \xi_{q_2}\big)\Big) \leq \frac14\|\xi_{q_1}\|^2\! +\frac14\|\xi_{q_2}\|^2\\
& \hspace{10.5cm} +C\|\xi_u\|^2+Ch^{2k+2}.
\end{aligned}
\end{align}
Then, we estimate the damping term $D_{ij}(u_h,\xi_u)$.
\begin{align*}
D_{ij}(u_h,\xi_u)&= -\sum_{\ell=0}^{k}\frac{\sigma_{K_{i, j}}^\ell(u_h)}{h_{K_{i, j}}}
\int_{K_{i, j}} \big( u_h-P_h^{\ell-1}u_h \big)\xi_u \, dxdy \\
&=-\sum_{\ell=0}^{k}\frac{\sigma_{K_{i, j}}^\ell(u_h)}{h_{K_{i, j}}}
\int_{K_{i, j}}\big(\xi_u-P_h^{\ell-1}\xi_u\big)^2
+ \big(\mathbb{G}_{\theta_1,\theta_2}\,u-P_h^{\ell-1}(\mathbb{G}_{\theta_1,\theta_2}\,u)\big) \xi_u\, dxdy  \\
&\leq -\sum_{\ell=0}^{k}\frac{\sigma_{K_{i, j}}^\ell(u_h)}{h_{K_{i, j}}}
\int_{K_{i, j}}\big(\mathbb{G}_{\theta_1,\theta_2}\,u-P_h^{\ell-1}(\mathbb{G}_{\theta_1,\theta_2}\,u)\big) \xi_u\, dxdy\\
&\leq \sum_{\ell=0}^{k}\frac{\sigma_{K_{i, j}}^\ell(u_h)}{h_{K_{i, j}}}
\big\|\mathbb{G}_{\theta_1,\theta_2}\,u-P_h^{\ell-1}(\mathbb{G}_{\theta_1,\theta_2}\,u) \big\|_{L^2({K_{i, j}})}
\|\xi_u\|_{L^2({K_{i, j}})}.
\end{align*}
Similar to the one-dimensional case, we need to estimate $\big\|\mathbb{G}_{\theta_1,\theta_2}\,u-P_h^{\ell-1}(\mathbb{G}_{\theta_1,\theta_2}\,u) \big\|_{L^2({K_{i, j}})}$ and $\sigma_j^\ell(u_h)$.
Thanks to the property of projections $P_h^{\ell-1}$ and $\mathbb{G}_{\theta_1,\theta_2}$, we get
\begin{align*}
& \|\mathbb{G}_{\theta_1,\theta_2}\,u-P_h^{\ell-1}(\mathbb{G}_{\theta_1,\theta_2}\,u)\|_{L^2({K_{i, j}})} \\
\leq & ~2\|\mathbb{G}_{\theta_1,\theta_2}\,u-u\|_{L^2({K_{i, j}})}+\|u-P_h^{\ell-1}u\|_{L^2({K_{i, j}})}\\
\leq & ~ Ch^{k+1}\|u\|_{H^{k+1}(\Omega)}+ h^{\max(1,\ell)+1}\|u\|_{W^{\max(1,\ell),\infty}(\Omega)}\\
\leq & ~Ch^{\max(1,\ell)+1}\|u\|_{H^{k+2}(\Omega)}.
\end{align*}
For the coefficients $\sigma_{K_{i, j}}^\ell(u_h)$, we have
\begin{align*}
\sum_{j=1}^{N_y}\sum_{i=1}^{N_x} \big( \sigma_{K_{i, j}}^\ell(u_h) \big)^2
&=\sum_{j=1}^{N_y}\sum_{i=1}^{N_x}\frac{4(2\ell+1)^2}{(2k-1)^2}\frac{h^{2\ell}}{(\ell!)^2}\sum_{|\bm{\alpha}|
=\ell}\Big(\frac{1}{N_e}\sum_{\bm{v}\in K_{i, j}}\big(\jump{\partial^{\bm{\alpha}}u_h
- \partial^{\bm{\alpha}}u}\Big|_{\bm{v}}\big)^2\Big) \\
&\leq C \sum_{j=1}^{N_y}\sum_{i=1}^{N_x}\sum_{|\bm{\alpha}|=\ell}\frac{h^{2\ell}}{N_e}\sum_{\bm{v}\in K_{i, j}}\Big(\jump{\partial^{\bm{\alpha}}\xi_u}^2\Big|_{\bm{v}}+\jump{\partial^{\bm{\alpha}}\eta_u}^2\Big|_{\bm{v}}\Big) \\
&\leq Ch^{-2}\|\xi_u\|^2+Ch^{2k} \, .
\end{align*}
Then, by the Cauchy-Schwarze inequality, we have
\begin{align} \label{damping2dest}
\begin{aligned}
D(u_h,\xi_u) = \sum_{i, j} D_{ij}(u_h, \xi_u) &\leq Ch\bigg(\sum_{j=1}^{N_y}\sum_{i=1}^{N_x}
\sum_{\ell=0}^{k}\big(\sigma_{K_{i, j}}^\ell(u_h)\big)^2\bigg)^{\frac 12}\|\xi_u\|\\
&\leq C\|\xi_u\|^2+Ch^{2k+2}.
\end{aligned}
\end{align}
Thus, combining \eqref{est1_2d}-\eqref{damping2dest}, we can obtain
\begin{align*}
\frac 12\frac{d}{dt}\|\xi_u\|^2+\|\xi_{q_1}\|^2+\|\xi_{q_2}\|^2\leq Ch^{2k+2}+C\|\xi_u\|^2
+\frac{1}{2}\|\xi_{q_1}\|^2+\frac{1}{2}\|\xi_{q_2}\|^2.
\end{align*}
After applying the Gronwall's inequality, we have
$\displaystyle\|\xi_u\|\leq Ch^{k+1}$. Finally, combining with the triangle inequality,
we obtain the optimal error estimate \eqref{error_est_2d}.
\end{proof}

\section{Numerical tests}
\label{sec_numeric}

In this section, we test some numerical examples to demonstrate the good performances
of the proposed scheme. We consider the one- and two-dimensional nonlinear degenerate parabolic equations.
Some strongly degenerate parabolic equations are also considered.
In all numerical tests, the time discretization employs the classic third order TVD Runge-Kutta method \cite{SO1988JCP}.
The space is uniformly divided in each direction.
Without special statement, the time step for one-dimensional problems \eqref{1d_model} is taken as
\begin{align}
\Delta t=\frac{CFL}{b/h^2+c/h},
\end{align}
with $b=\max_{u}|a(u)|$, $c=\max_{u}|f'(u)|$ and $CFL=0.1$. For two-dimensional problems \eqref{2d_model}, we take
\begin{align}
\Delta t=\frac{CFL}{b_x/h_x^2+b_y/h_y^2 +c_x/h_x+c_y/h_y},
\end{align}
with $b_x=\max_{u}|a_1(u)|$, $b_y=\max_{u}|a_2(u)|$, $c_x=\max_u|f_1'(u)|$, $c_y=\max_u|f_2'(u)|$
and $CFL=0.1$. We employ the piecewise $\mathcal{P}^2$ polynomial space to simulate all numerical tests
unless otherwise specified. The cell averages are plotted to show the numerical solutions in our test.
We also emphasize that no limiter is used in all simulations here.


\begin{exmp}\label{examp1}
The first example is the Barenblatt solution of the porous medium equation (PME), namely,
\begin{align}\label{PME_1d}
u_t=(u^m)_{xx}, \quad x \in \mathbb{R},\quad t>0,
\end{align}
where $m$ is a constant greater than one. The Barenblatt solution of PME \eqref{PME_1d} is defined by
\begin{align}\label{eq:Barenblatt}
B_{m}(x,t)=t^{-p}\left[\left(1-\frac{p(m-1)}{2m}\frac{|x|^2}{t^{2p}}\right)_{+}\right]^{1/(m-1)},
\end{align}
where $u_{+}= \max\{u,0\}$ and $p=(m+1)^{-1}$. The solution has a compact support $[-\alpha_{m}(t),\alpha_m(t)]$ with
\begin{align*}
\alpha_{m}(t)=t^p\sqrt{\frac{2m}{p(m-1)}},
\end{align*}
and the interface $|x|=\alpha_m(t)$ moving outward in a finite speed. For this problem,
\begin{align}
b(u)=\sqrt{mu^{m-1}}, \quad g(u)=\frac{2 u \sqrt{m u^{m-1}}}{1+m}.
\end{align}

\noindent
We take the initial solution as the Barenblatt solution at $t=1$. Consider the domain $I=[-6,6]$
with the boundary condition $u(\pm 6,t)=0$ for $t\geq1$. The numerical solution is obtained at $t=2$.
\end{exmp}
We plot the numerical solutions with $N=320$ grid points for $m=2,3,5$ and $8$ in Figure \ref{fig1d_PME}, respectively. We can clearly observe that the numerical solutions accurately capture the interface $|x|=\alpha_m(t)$ without noticeable oscillations.

\begin{figure}[htb]
\centering
\caption{ \label{fig1d_PME} Example \ref{examp1}: Barenblatt solution for the PME with  grid points $N=320$.}
\subfigure[$m=2$]{
\includegraphics[width=0.45\textwidth]{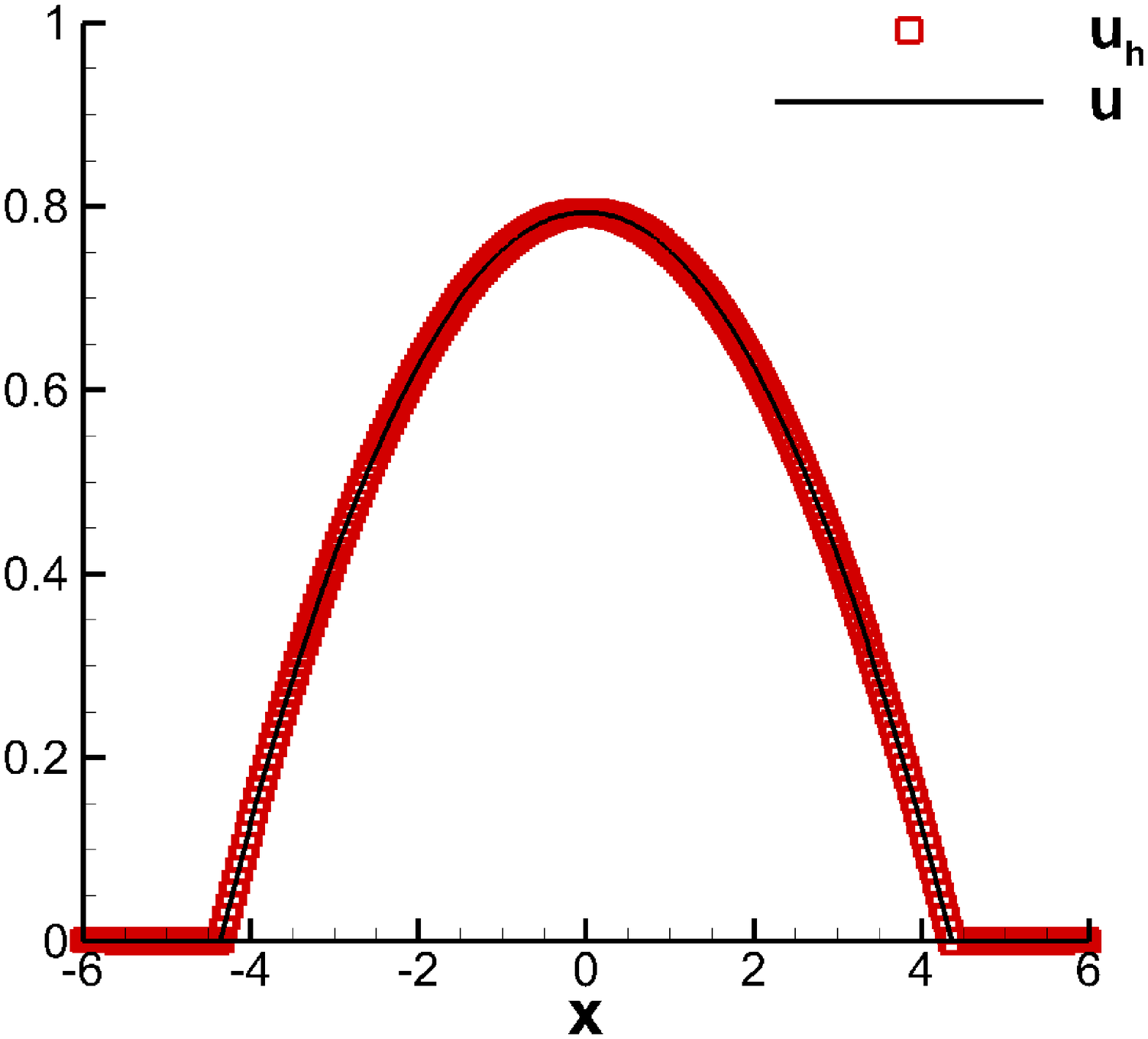} }
\subfigure[$m=3$]{
\includegraphics[width=0.45\textwidth]{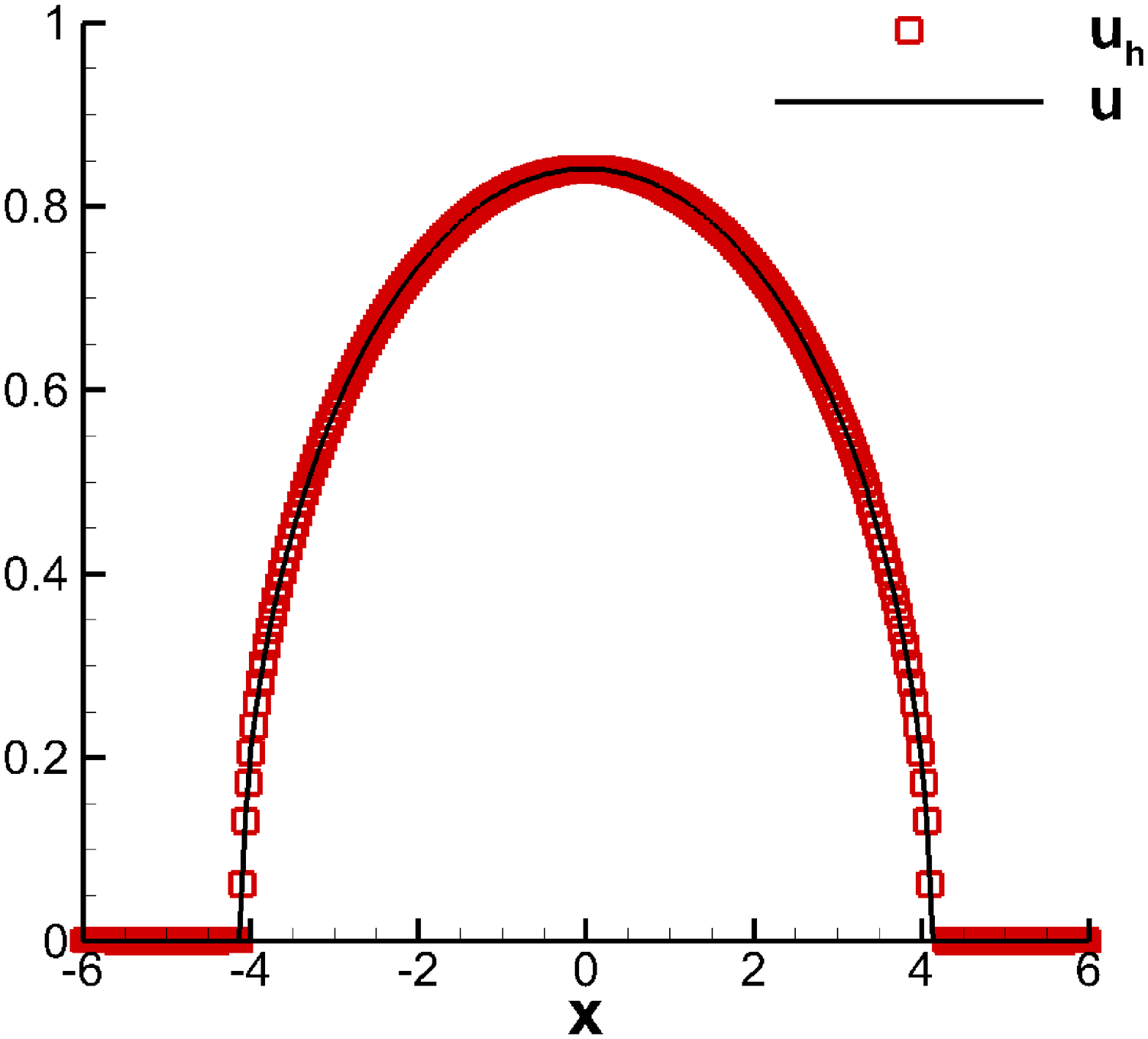} }
\subfigure[$m=5$]{
\includegraphics[width=0.45\textwidth]{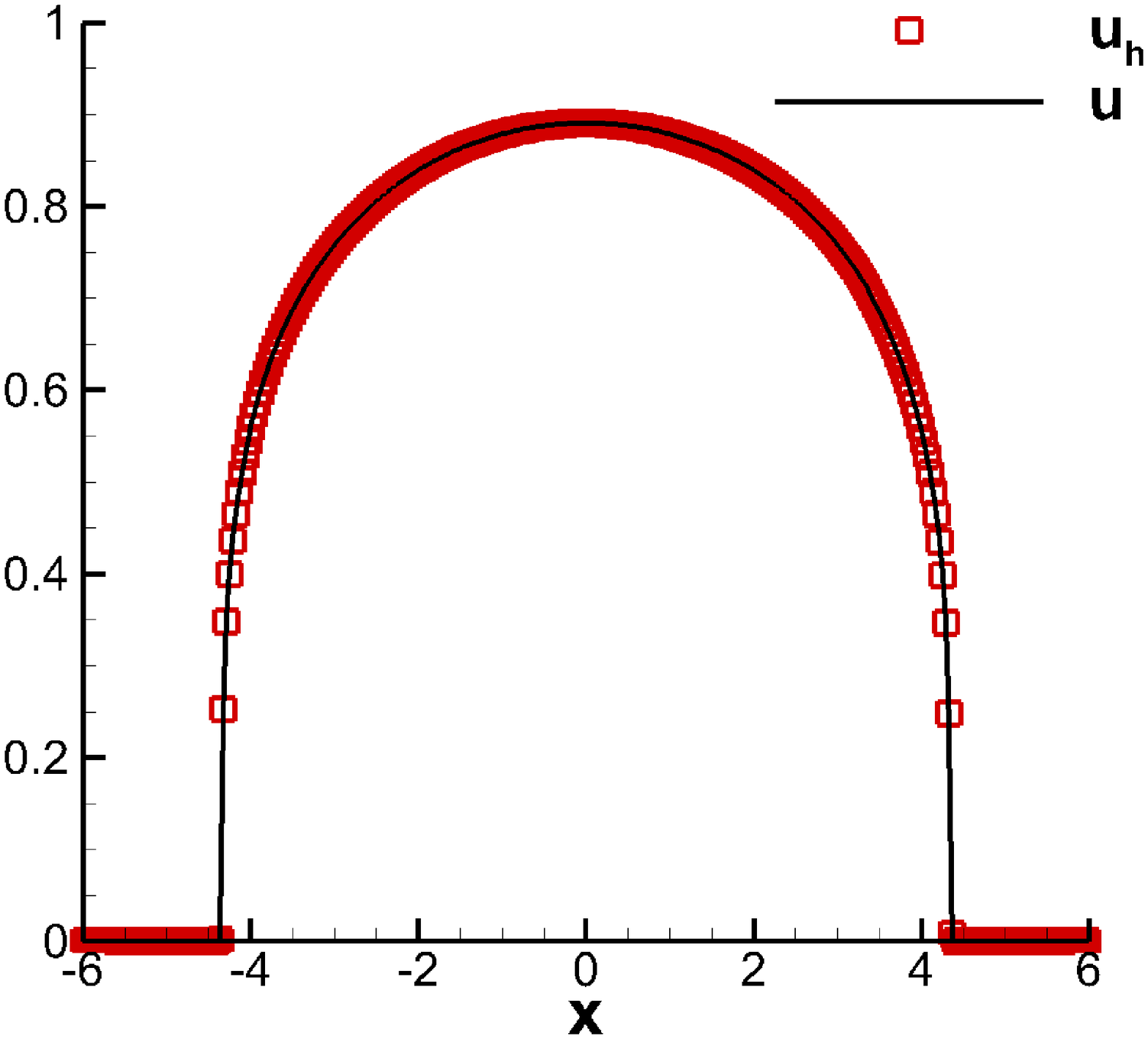} }
\subfigure[$m=8$]{
\includegraphics[width=0.45\textwidth]{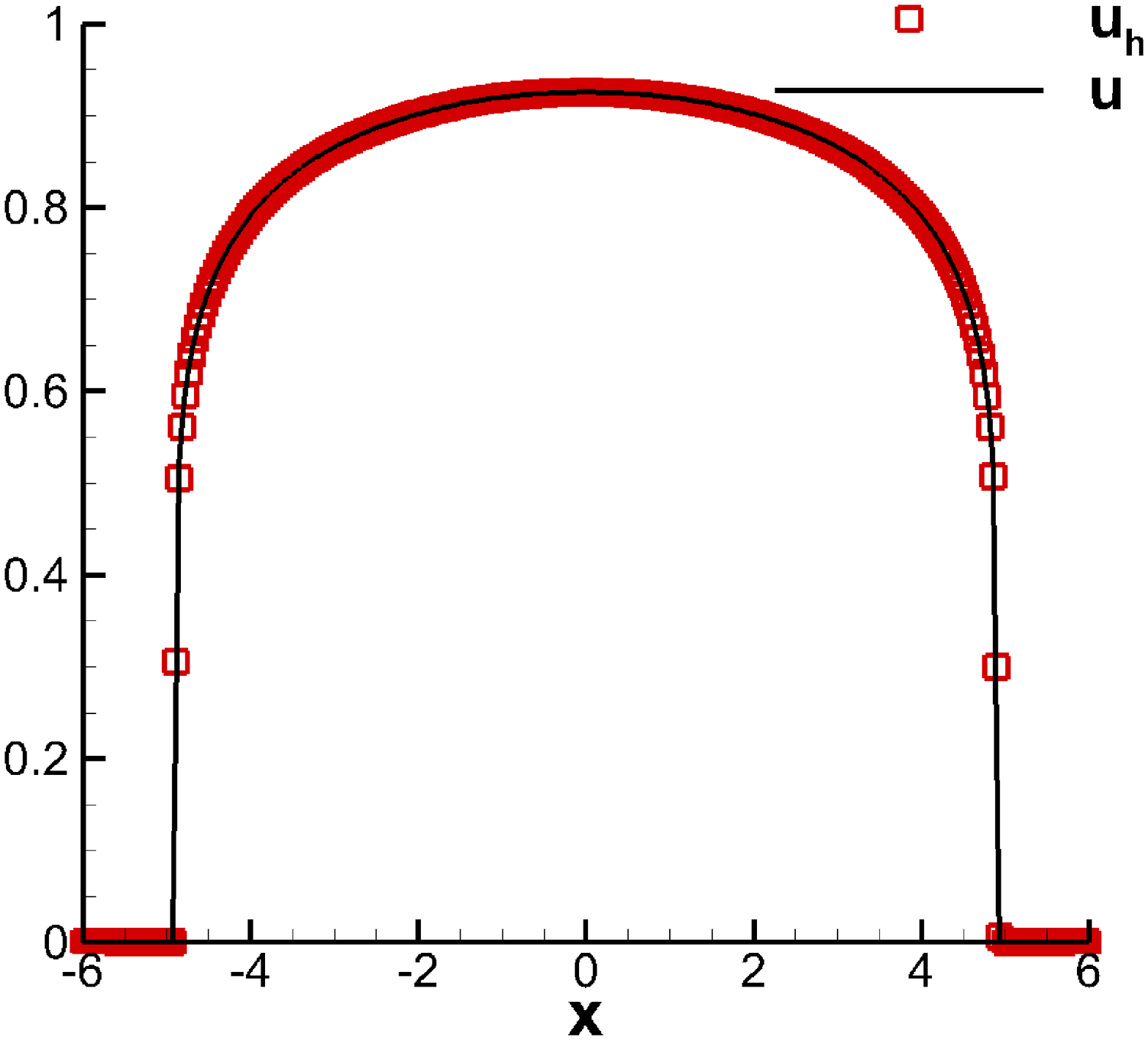} }
\end{figure}

We also test the accuracy in the smooth part of the solution. We compute the error for the Barenblatt solution \eqref{eq:Barenblatt} of the PME with $m=8$ on domain $[-1.5,1.5]$ which is the smooth part of the solution and our final time is $t=1.05$. From Table \ref{tab1}, the optimal order of error is observed for this problem.


\begin{table}[htb]\small
\caption{\label{tab1} Example \ref{examp1}: The errors and orders of $u_h$ for the smooth part of the Barenblatt solution.} \centering
\medskip
\begin{tabular}{|c|r||cc|cc|cc|}  \hline
 & $N$ & $L^1$ error & order & $L^2$ error & order & $L^\infty$ error & order \\ \hline \hline
 \multirow{5}{0.6cm}{$\mathcal{P}^1$}
     &   40    &  2.001E-04    &  --    &  1.551E-04    &  --    &  2.206E-04    &  --   \\
    &   80    &  4.697E-05    &   2.091    &  3.726E-05    &   2.058    &  5.252E-05    &   2.070   \\
    &  160    &  1.139E-05    &   2.044    &  9.139E-06    &   2.027    &  1.284E-05    &   2.032   \\
    &  320    &  2.805E-06    &   2.022    &  2.264E-06    &   2.013    &  3.177E-06    &   2.015   \\
    &  640    &  6.958E-07    &   2.011    &  5.633E-07    &   2.007    &  7.901E-07    &   2.007   \\\hline
 \multirow{5}{0.6cm}{$\mathcal{P}^2$}
     &   40    &  1.256E-06    &  --    &  1.027E-06    &  --    &  1.748E-06    &  --   \\
    &   80    &  1.388E-07    &   3.177    &  1.149E-07    &   3.160    &  2.018E-07    &   3.115   \\
    &  160    &  1.633E-08    &   3.087    &  1.363E-08    &   3.076    &  2.424E-08    &   3.057   \\
    &  320    &  1.981E-09    &   3.043    &  1.661E-09    &   3.037    &  2.968E-09    &   3.030   \\
    &  640    &  2.440E-10    &   3.022    &  2.050E-10    &   3.018    &  3.667E-10    &   3.017   \\\hline
 \multirow{4}{0.6cm}{$\mathcal{P}^3$}
        &   40    &  3.463E-08    &  --    &  2.527E-08    &  --    &  5.813E-08    &  --   \\
    &   80    &  1.785E-09    &   4.278    &  1.085E-09    &   4.541    &  1.495E-09    &   5.281   \\
    &  160    &  1.063E-10    &   4.069    &  6.436E-11    &   4.076    &  8.226E-11    &   4.184   \\
    &  320    &  6.499E-12    &   4.032    &  3.940E-12    &   4.030    &  4.852E-12    &   4.084   \\\hline
\end{tabular}
\end{table}

\begin{exmp}
\label{examp2}
Next, we consider the interaction of tow boxes for the PME \eqref{PME_1d}. We take the initial data as
\begin{align}
u(x,0)=\left\{\begin{array}{l}
1, \quad x\in (-4,-1),\\
1.5, \quad x\in (0,3),\\
0,\quad \text{otherwise},
\end{array}\right.
\end{align}
and the computational domain $I=[-6,6]$ with boundary condition $u(\pm 6,t)=0$. The uniform mesh with $N=240$ cells
is used to compute until the terminal time $t=1.0$. The parameter $m=8$.
\end{exmp}
We show the evolution of the numerical solution at different time in Figure \ref{fig_1d_PME_two_box}. From the results,
we can observe that the numerical solutions don't appear noticeable oscillation around the interface and agree very well with the reference solution in \cite{Jiang2021JSC, Liu2011SISC}.

\begin{figure}[p]
\centering
\caption{ \label{fig_1d_PME_two_box} Example \ref{examp2}: Interaction of tow boxes for the PME with grid points $N=240$.}
\subfigure[$t=0.$]{
\includegraphics[width=0.45\textwidth,height=0.15\textheight]{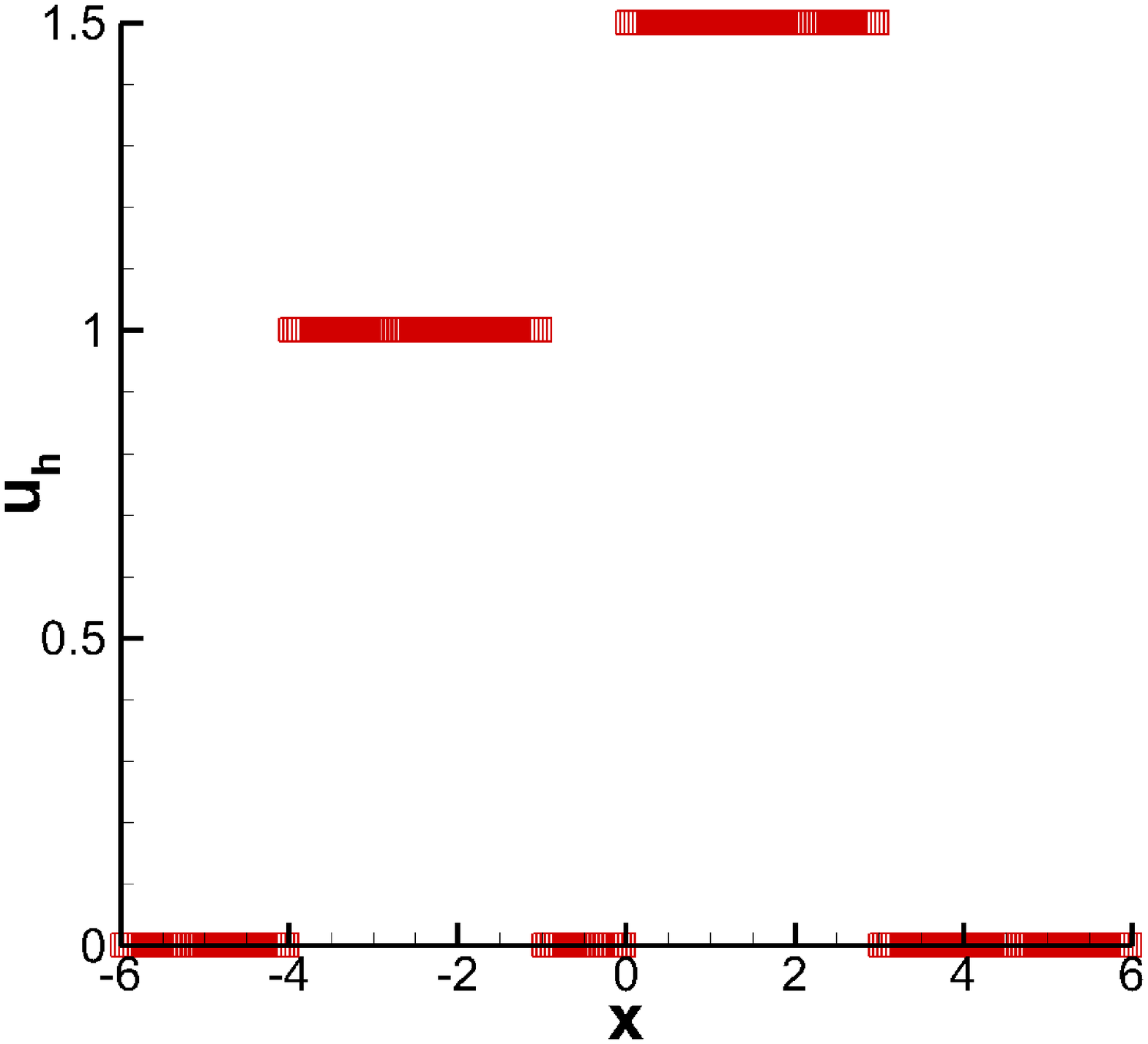} }
\subfigure[$t=0.05$]{
\includegraphics[width=0.45\textwidth,height=0.15\textheight]{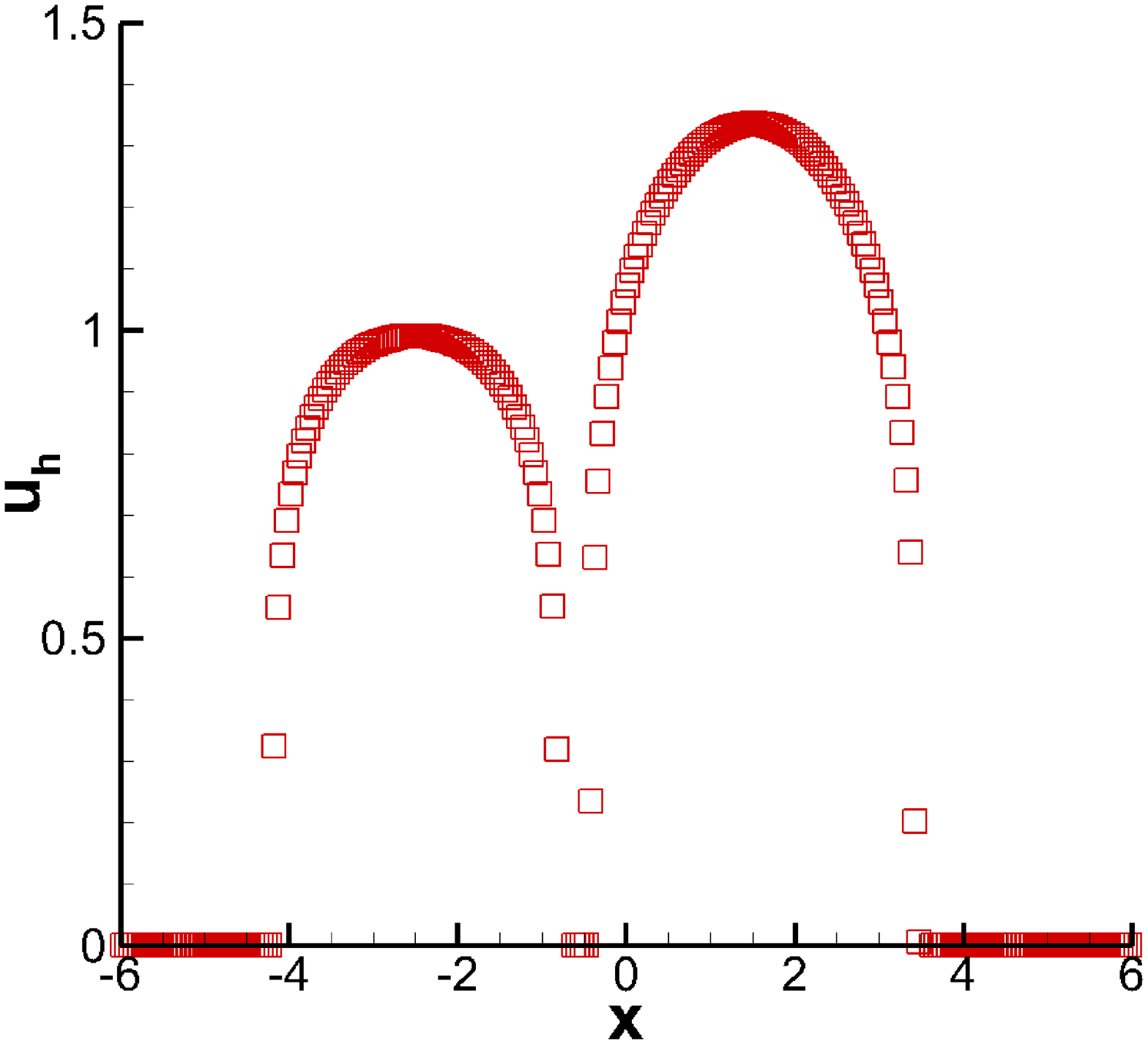} }
\subfigure[$t=0.08$]{
\includegraphics[width=0.45\textwidth,height=0.15\textheight]{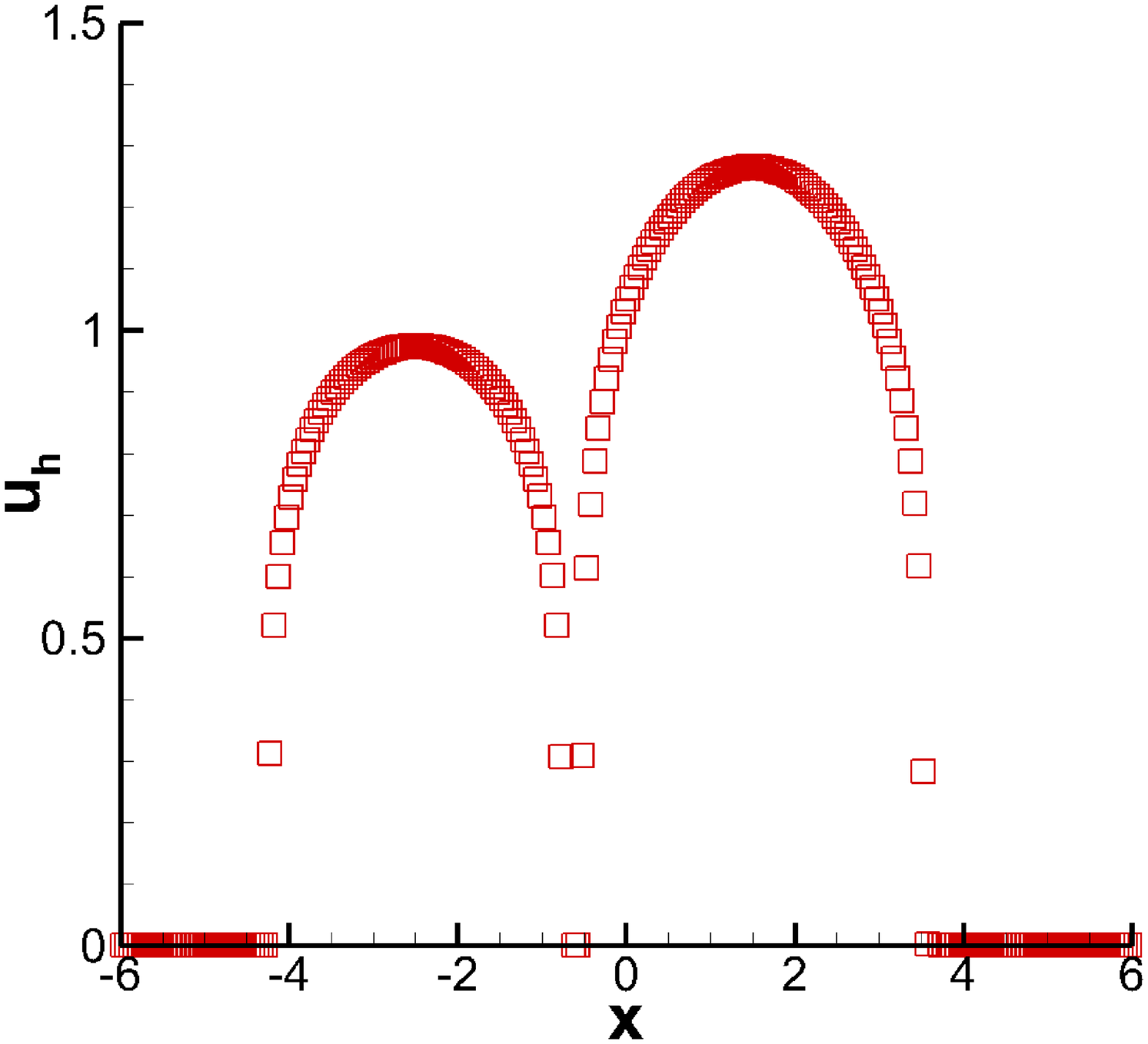} }
\subfigure[$t=0.11$]{
\includegraphics[width=0.45\textwidth,height=0.15\textheight]{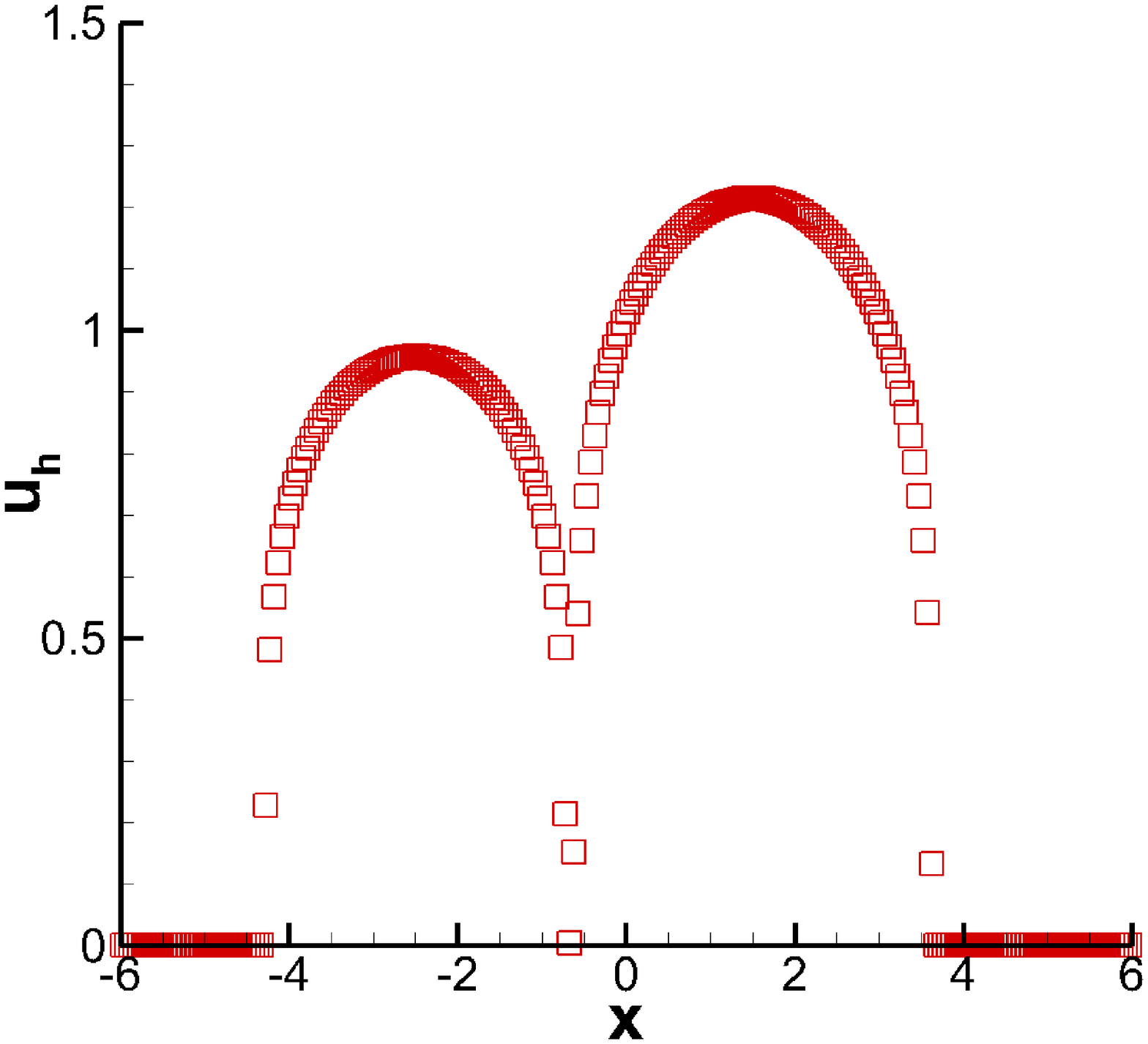} }
\subfigure[$t=0.14$]{
\includegraphics[width=0.45\textwidth,height=0.15\textheight]{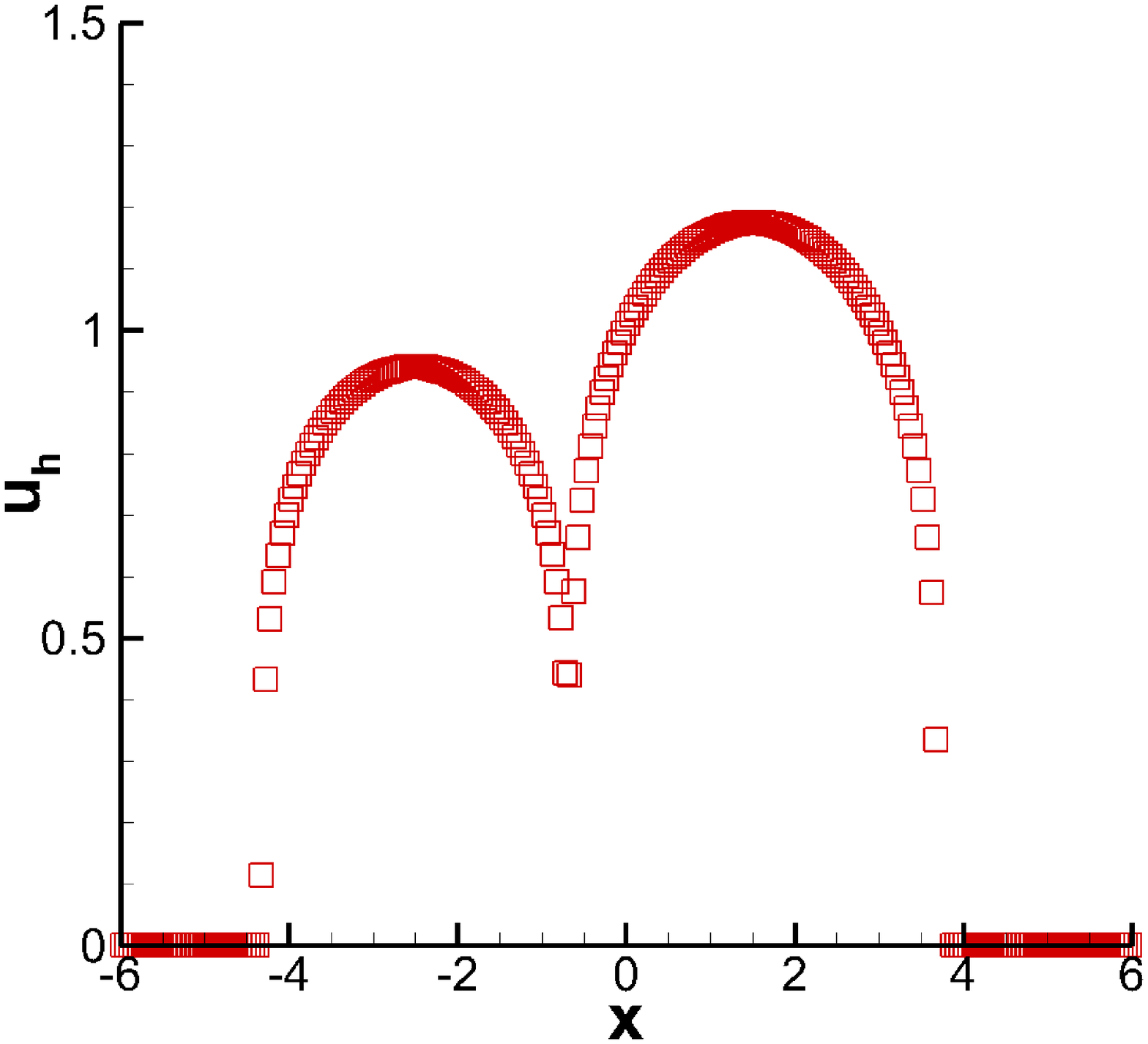} }
\subfigure[$t=0.17$]{
\includegraphics[width=0.45\textwidth,height=0.15\textheight]{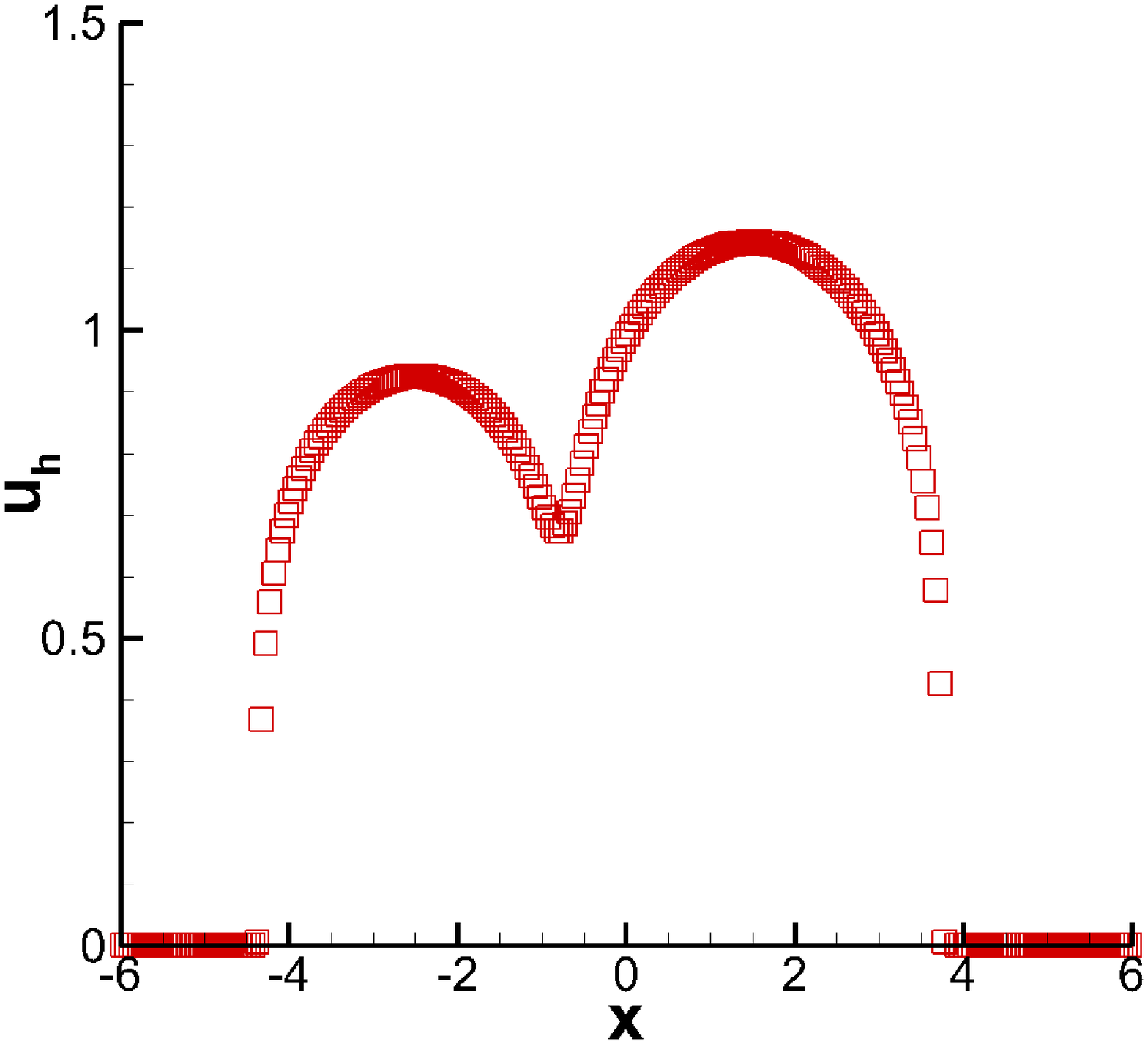} }
\subfigure[$t=0.20$]{
\includegraphics[width=0.45\textwidth,height=0.15\textheight]{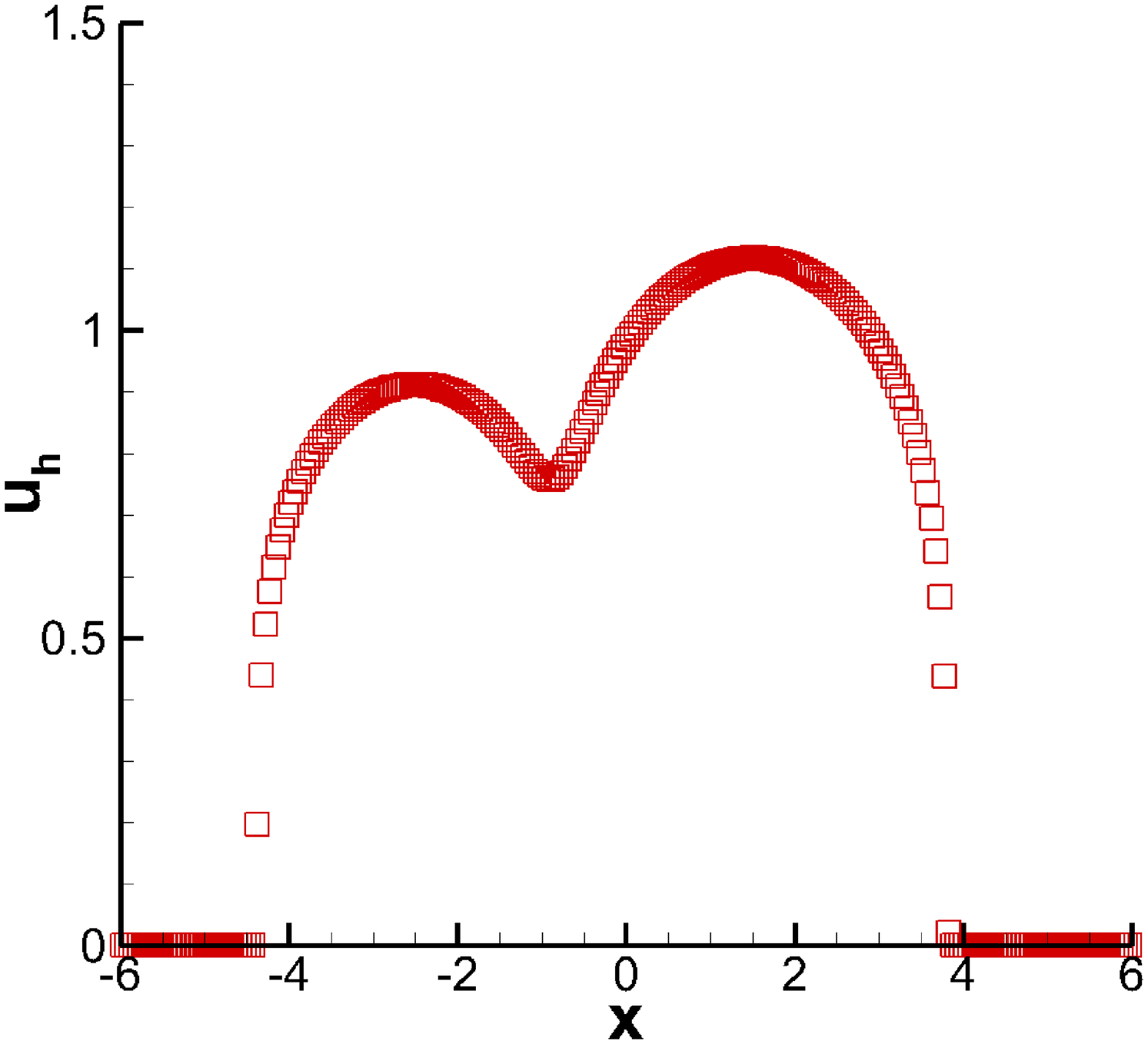} }
\subfigure[$t=0.23$]{
\includegraphics[width=0.45\textwidth,height=0.15\textheight]{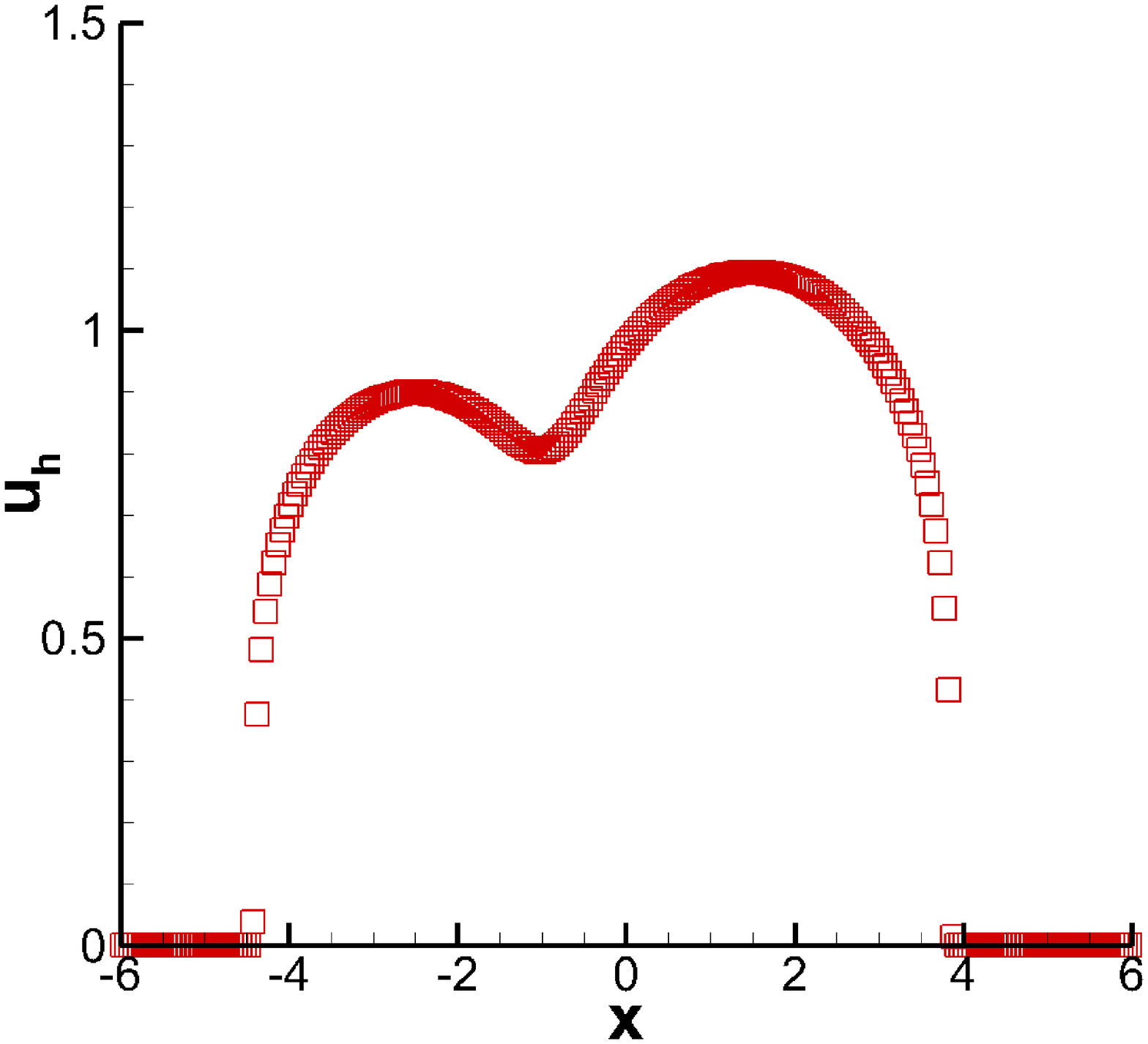} }
\subfigure[$t=0.50$]{
\includegraphics[width=0.45\textwidth,height=0.15\textheight]{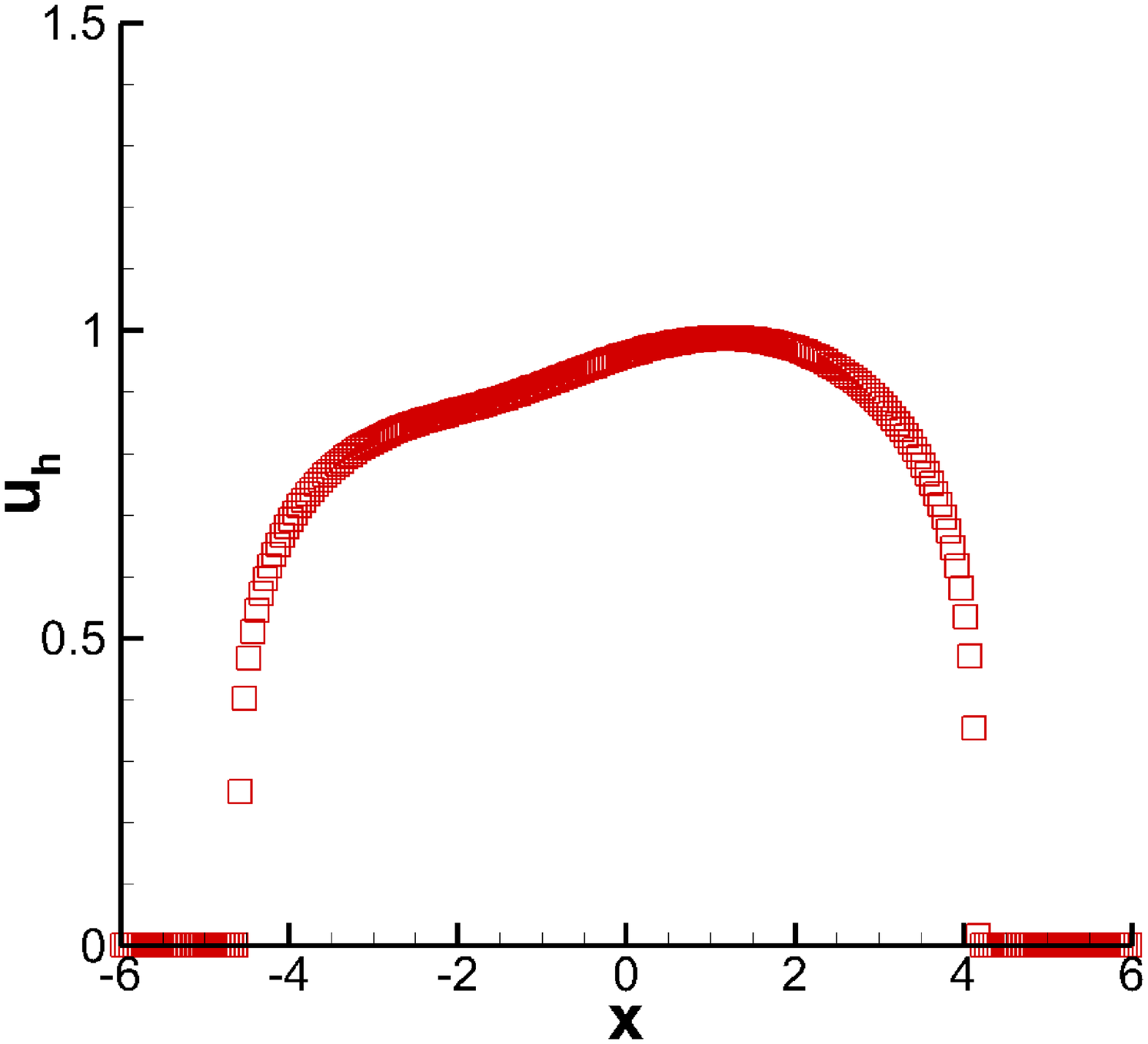} }
\subfigure[$t=1.00$]{
\includegraphics[width=0.45\textwidth,height=0.15\textheight]{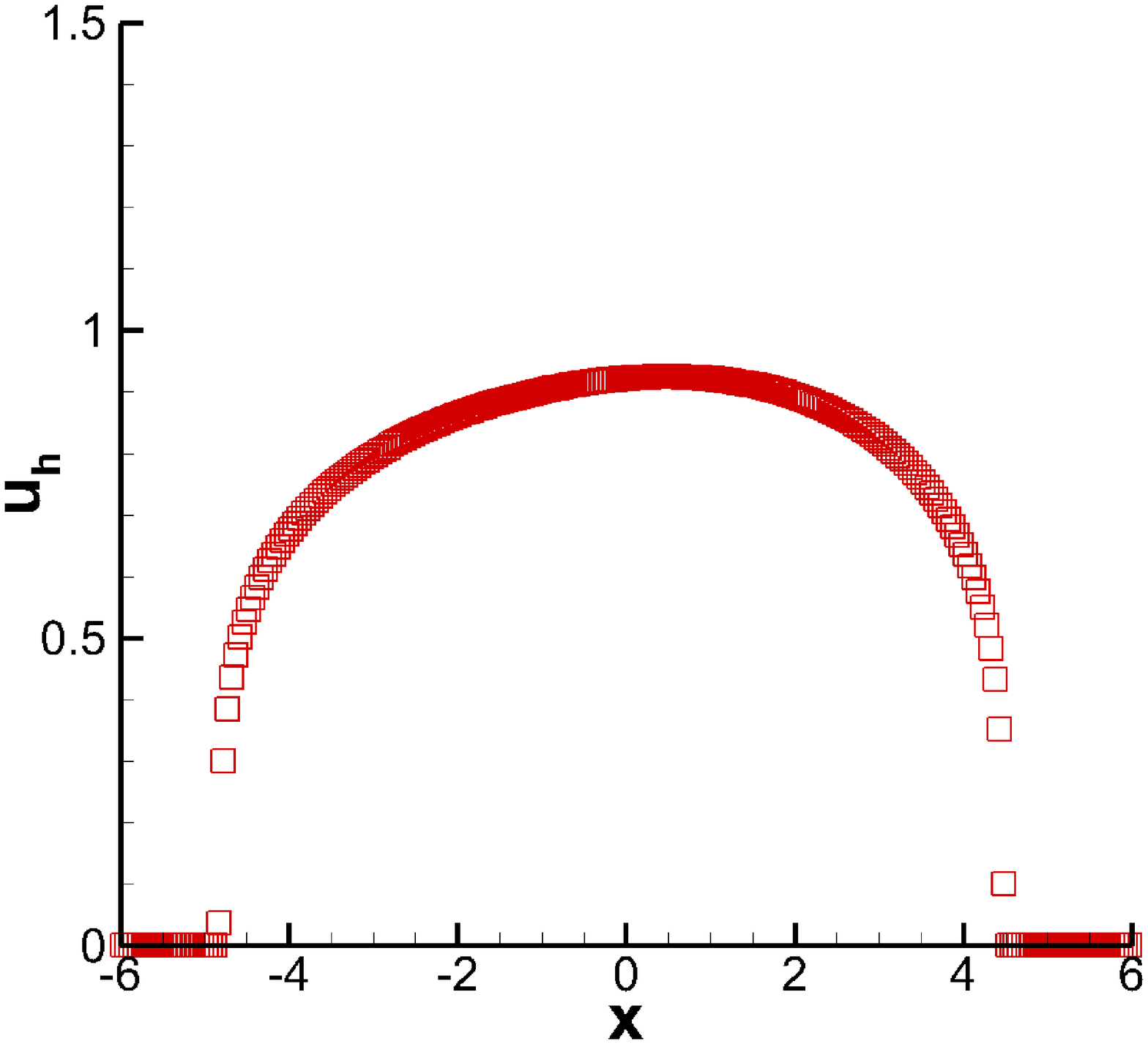} }
\end{figure}

\begin{exmp}\label{exmp3}
In this example, let us consider the Buckley-Leverett equation \cite{Buckley1942AIME}
\begin{align}
u_t+f(u)_x=\epsilon(\nu(u)u_x)_x, \quad x\in[0,1],
\end{align}
which is usually used to model two-phase flow in porous media in fluid dynamics, such as displacing oil by water in a one-dimensional or quasi-one-dimensional reservoir. We choose the parameter $\epsilon=0.01$ and
\begin{align}
\nu(u)=\left\{\begin{array}{ll}
4u(1-u), & 0 \leq u \leq 1,\\
0,&\text{otherwise}.
\end{array}
\right.
\end{align}
So, we can get
\begin{align}
&a(u)=\epsilon\nu(u),\\
&g(u)=\int^u\sqrt{a(u)}\, du =\left\{\begin{array}{ll}
0, & u<0,\\
\sqrt{\epsilon}(\frac{\theta}{2}-\frac{1}{8}\sin(4\theta)), & u=\sin^2(\theta),\quad 0\leq \theta \leq \frac{\pi}{2}.\\
\frac{\sqrt{\epsilon}\pi}{4},& u>1.
\end{array}\right.
\end{align}
We will consider two kinds of flux functions. One is no gravitational effects
\begin{align}\label{flux_no_gra}
f(u)=\frac{u^2}{u^2+(1-u)^2},
\end{align}
the other has gravitational effects
\begin{align}\label{flux_gra}
f(u)=\frac{u^2}{u^2+(1-u)^2}(1-5(1-u)^2).
\end{align}
\end{exmp}
We firstly consider the flux \eqref{flux_no_gra} and take the initial condition as
\begin{align}
u(x,0)=\left\{\begin{array}{ll}
1-3x,& 0\leq x \leq 1/3,\\
0, & 1/3 < x \leq 1.
\end{array}\right.
\end{align}
The boundary conditions $u(0,t)=1$ and $u(1,t)=0$ are imposed. Our terminal time is $t=0.2$. We test this example with different number of cells, the numerical results are shown in Figure \ref{fig_1d_BL_exmp1}.
It indicates the numerical solution converges to the entropy solution as the mesh refining.

\begin{figure}[htb]
\centering
\caption{ \label{fig_1d_BL_exmp1} Example \ref{exmp3}: Initial-boundary value problem for the Buckley-Leverett equation.}
\includegraphics[width=0.45\textwidth]{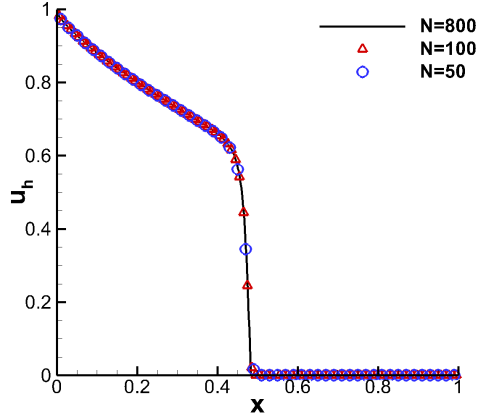}
\end{figure}

Secondly, we solve a Riemann problem with both fluxes \eqref{flux_no_gra} and \eqref{flux_gra}. The initial data is given as
\begin{align}
u(x,0)=\left\{\begin{array}{ll}
0,&0\leq x <1 - \frac{1}{\sqrt{2}},\\
1,& 1-\frac{1}{\sqrt{2}} \leq x \leq 1.
\end{array}\right.
\end{align}
The terminal time is $t=0.2$. The numerical results are shown in Figure \ref{fig_1d_BL_exmp2}. We can observe that the scheme can sharply capture the contacts without noticeable spurious oscillations, and the results are benchmarked against those in \cite{Jiang2021JSC}.
\begin{figure}[htb]
\centering
\caption{ \label{fig_1d_BL_exmp2} Example \ref{exmp3}: Riemann problems for the Bucklay-Leverett equation.}
\subfigure[without gravitation]{
\includegraphics[width=0.45\textwidth]{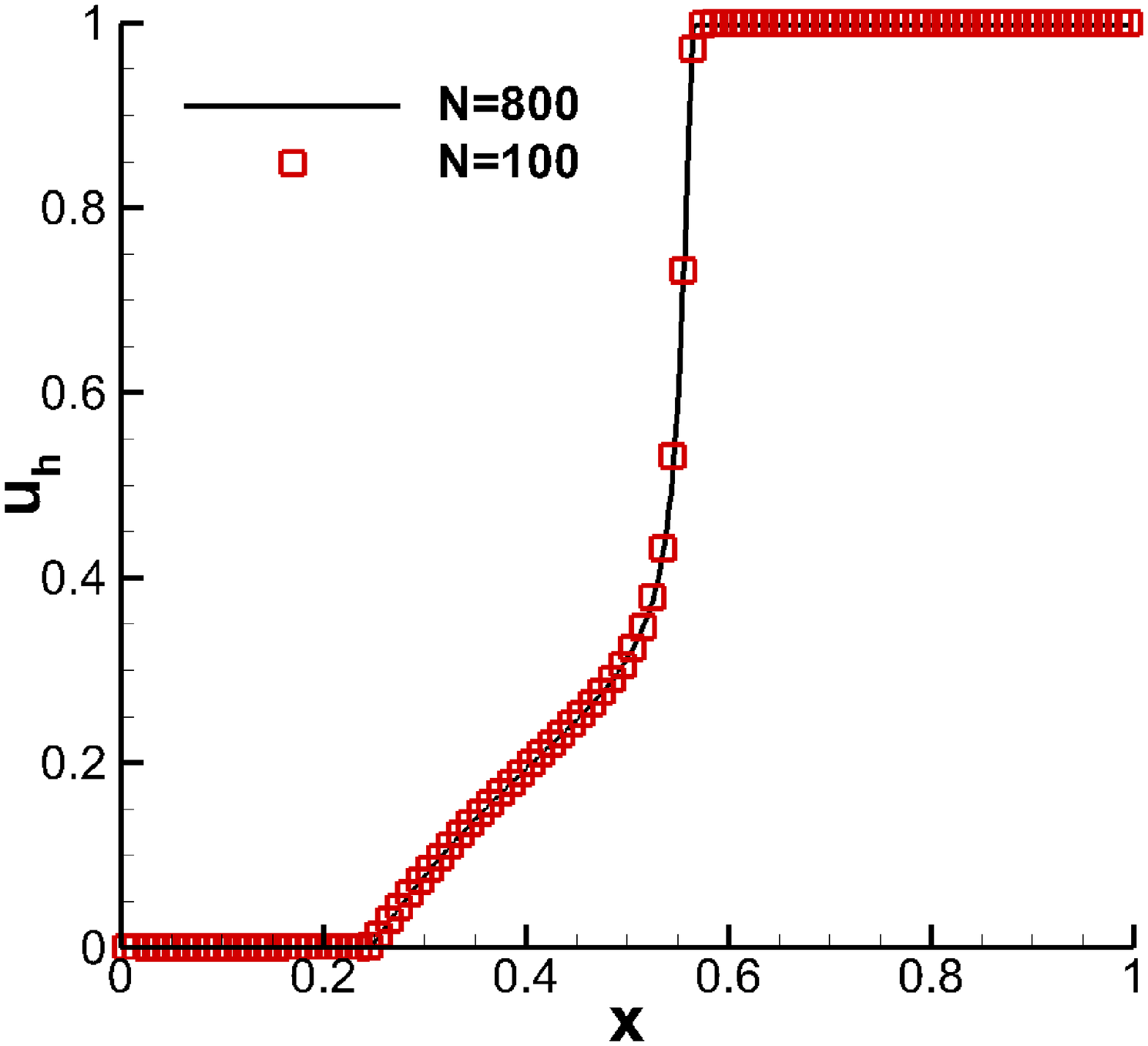} }
\subfigure[with gravitation]{
\includegraphics[width=0.45\textwidth]{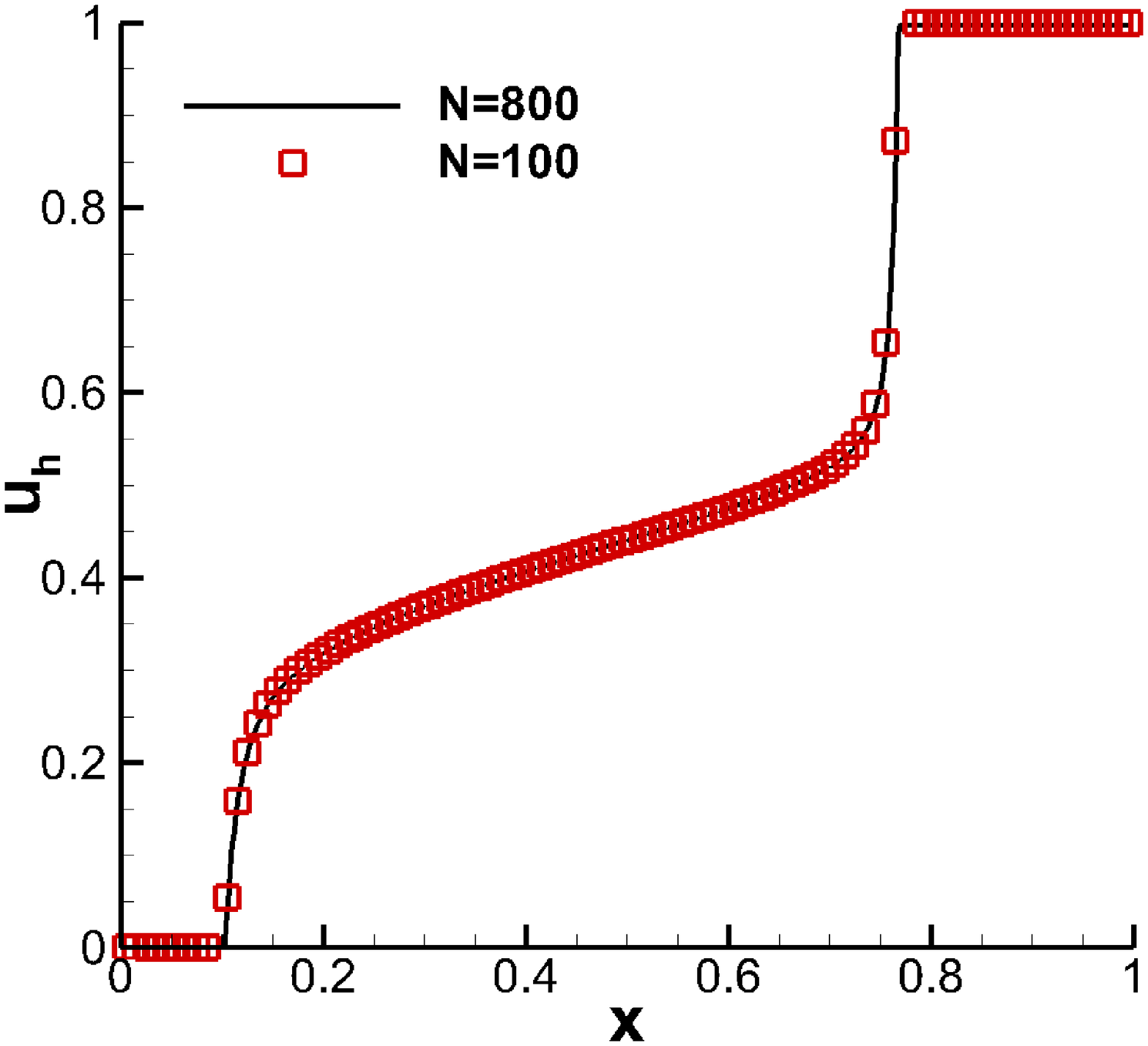} }
\end{figure}

\begin{exmp}\label{examp4}
Our final one-dimensional problem is a strongly degenerate parabolic convection-diffusion equation
\begin{align}
u_t+f(u)_x=\epsilon(\nu(u)u_x)_x,
\end{align}
with $\epsilon=0.1$, $f(u)=u^2$, and
\begin{align}
\nu(u)=\left\{\begin{array}{ll}
0,&|u|\leq 0.25,\\
1,&|u| >0.25.
\end{array}\right.
\end{align}
This $\nu$ will lead to the equation has hyperbolic property when $u\in [-0.25,0.25]$ and becomes parabolic elsewhere. We have,
\begin{align}
a(u)=\epsilon\nu(u),\quad g(u)=\int^u\sqrt{a(u)}\, du =\left\{\begin{array}{ll}
\sqrt{\epsilon}(u+0.25), & u<-0.25,\\
\sqrt{\epsilon}(u-0.25), & u>0.25,\\
0,& |u|\leq 0.25.
\end{array}\right.
\end{align}
We consider the following initial condition
\begin{align}
u(x,0)=\left\{\begin{array}{ll}
1,&-\frac{1}{\sqrt{2}}-0.4 < x < -\frac{1}{\sqrt{2}}+0.4,\\
-1,& \frac{1}{\sqrt{2}}-0.4 < x < \frac{1}{\sqrt{2}}+0.4,\\
0,&\text{otherwise},
\end{array}\right.
\end{align}
and a zero boundary condition $u(\pm 2, t)=0$, the final time is $t=0.7$.
\end{exmp}
In this example, we make a comparison. We solve this problem by using the original LDG scheme and our OFLDG scheme, the numerical results are provided in Figure \ref{fig_1d_strong_DP}. It can be clearly seen that
the spurious oscillations do appear in the numerical results without the damping terms, i.e. the original LDG scheme. However, the OFLDG scheme effectively controls the spurious oscillations and accurately captures the sharp interface. This indicates that the damping terms do have the ability of reducing the spurious oscillations of numerical solutions.
\begin{figure}[htb]
\centering
\caption{ \label{fig_1d_strong_DP} Example \ref{examp4}: Riemann problem for the strongly degenerate parabolic equation.}
\subfigure[without damping terms]{
\includegraphics[width=0.45\textwidth]{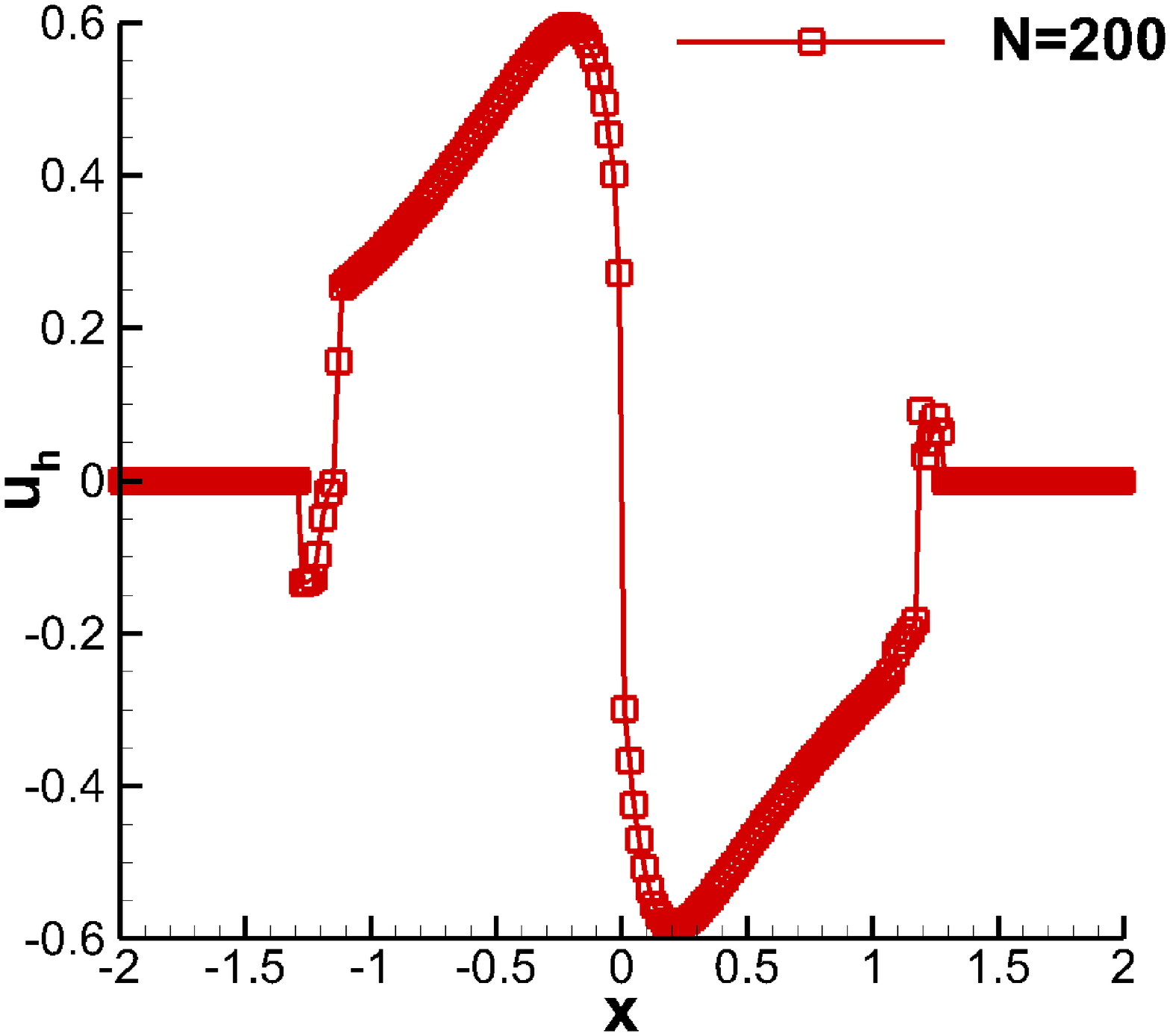}}
\subfigure[with damping terms]{
\includegraphics[width=0.45\textwidth]{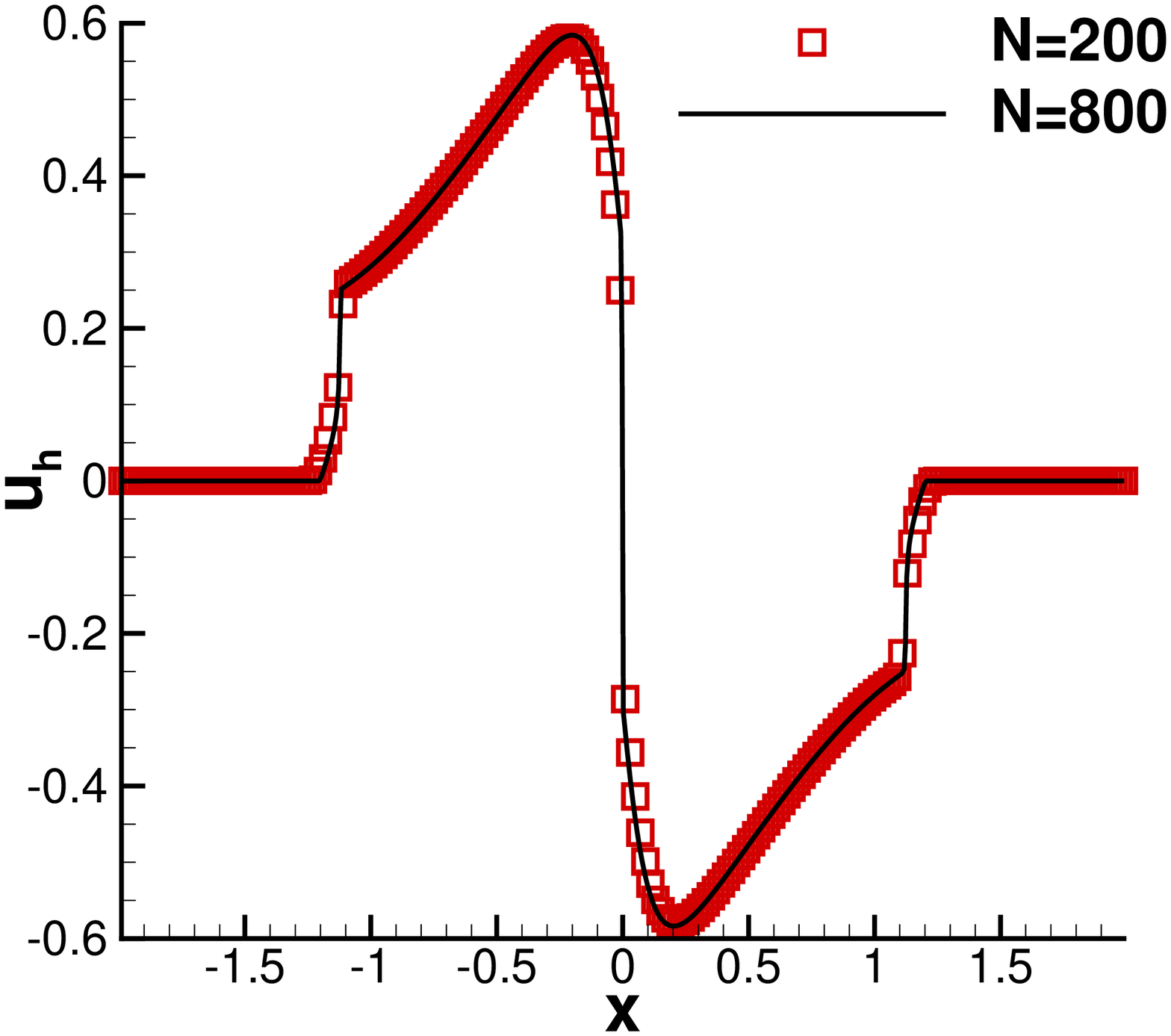}}
\end{figure}

\begin{exmp}\label{examp7}
Our first two-dimensional example is to test accuracy of the OFLDG method. Let's consider the heat equation
\begin{align}
\left\{\begin{array}{ll}
u_t=u_{xx}+u_{yy}, & (x,y) \in [-\pi,\pi]^2,\\
u(x,y,0)=\sin(x+y),&
\end{array}\right.
\end{align}
with $2\pi$-periodic boundary conditions in both directions. The exact solution of this problem is $u(x,y,t)=e^{-2t}\sin(x+y)$.
\end{exmp}
\begin{table}[htb]\small
\caption{\label{tab2} Example \ref{examp7}: Errors and orders of $u_h$ of the heat equation.} \centering
\medskip
\begin{tabular}{|c|r||cc|cc|cc|}  \hline
 & $N_x\times N_y$ & $L^1$ error & order & $L^2$ error & order & $L^\infty$ error & order \\ \hline \hline
 \multirow{4}{0.6cm}{$\mathcal{P}^1$ }
      &   $10\times10$    &  1.292E-03    &   --    &  1.495E-03    &   --   &  2.538E-03    &   --   \\
    &   $20\times20$    &  1.972E-04    &   2.711    &  2.367E-04    &   2.658    &  5.968E-04    &   2.089   \\
    &   $40\times40$    &  3.525E-05    &   2.484    &  4.419E-05    &   2.422    &  1.841E-04    &   1.697   \\
    &   $80\times80$    &  7.537E-06    &   2.225    &  9.953E-06    &   2.150    &  5.010E-05    &   1.877   \\
\hline
 \multirow{4}{0.6cm}{$\mathcal{P}^2$}
     &   $10\times10$   &  1.971E-04    &  --    &  2.188E-04    &  --   &  4.428E-04    &  --  \\
    &   $20\times20$    &  1.456E-05    &   3.759    &  1.655E-05    &   3.725    &  6.096E-05    &   2.861   \\
    &   $40\times40$    &  1.158E-06    &   3.653    &  1.411E-06    &   3.552    &  7.766E-06    &   2.973   \\
    &   $80\times80$    &  1.102E-07    &   3.392    &  1.475E-07    &   3.258    &  9.749E-07    &   2.994   \\
\hline
 \multirow{4}{0.6cm}{$\mathcal{P}^3$}
        &   $10\times10$    &  1.107E-05    &  --    &  1.344E-05    &  --    &  6.057E-05    &  --   \\
    &   $20\times20$    &  3.284E-07    &   5.075    &  4.798E-07    &   4.807    &  3.300E-06    &   4.198   \\
    &   $40\times40$    &  1.540E-08    &   4.415    &  2.403E-08    &   4.320    &  1.929E-07    &   4.096   \\
    &   $80\times80$    &  9.247E-10    &   4.058    &  1.411E-09    &   4.090    &  1.165E-08    &   4.050   \\
 \hline
\end{tabular}
\end{table}

The errors and the associated orders of $u_h$ at time $t=2$ are provided in Table \ref{tab2}. From Table \ref{tab2}, we can observe that the numerical solutions still have the optimal convergence order
when using the piecewise $\mathcal{P}^k$ finite element space for the rectangular meshes, although our theoretical results are based on piecewise $\mathcal{Q}^k$ finite element space.
It indicates that the damping terms does not reduce the accuracy of the original LDG scheme.
\begin{figure}[htb]
\centering
\caption{ \label{fig_2d_PME} Example \ref{exmap5}: The 2D PME. $N_x\times N_y=80\times 80$.}
\subfigure[$t=0$]{
\includegraphics[width=0.45\textwidth]{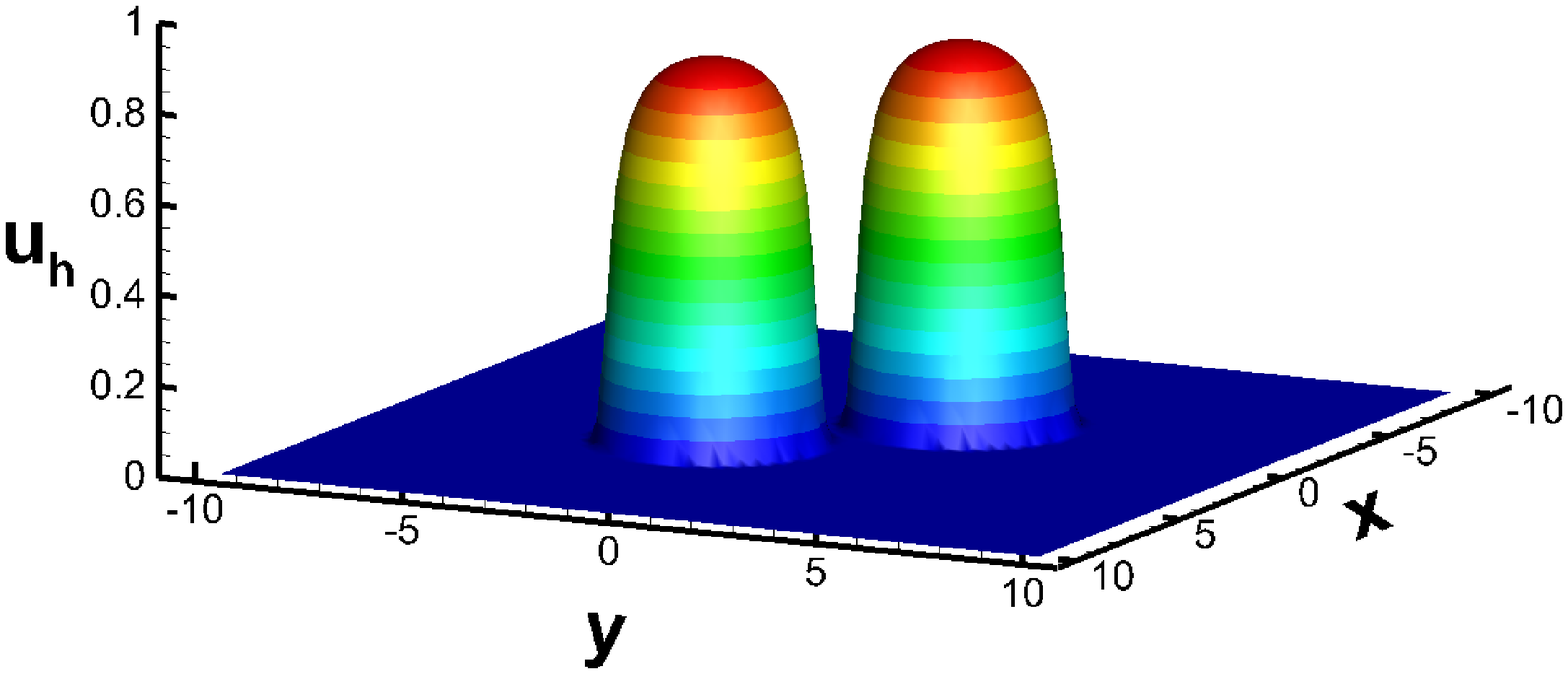} }
\subfigure[$t=0.5$]{
\includegraphics[width=0.45\textwidth]{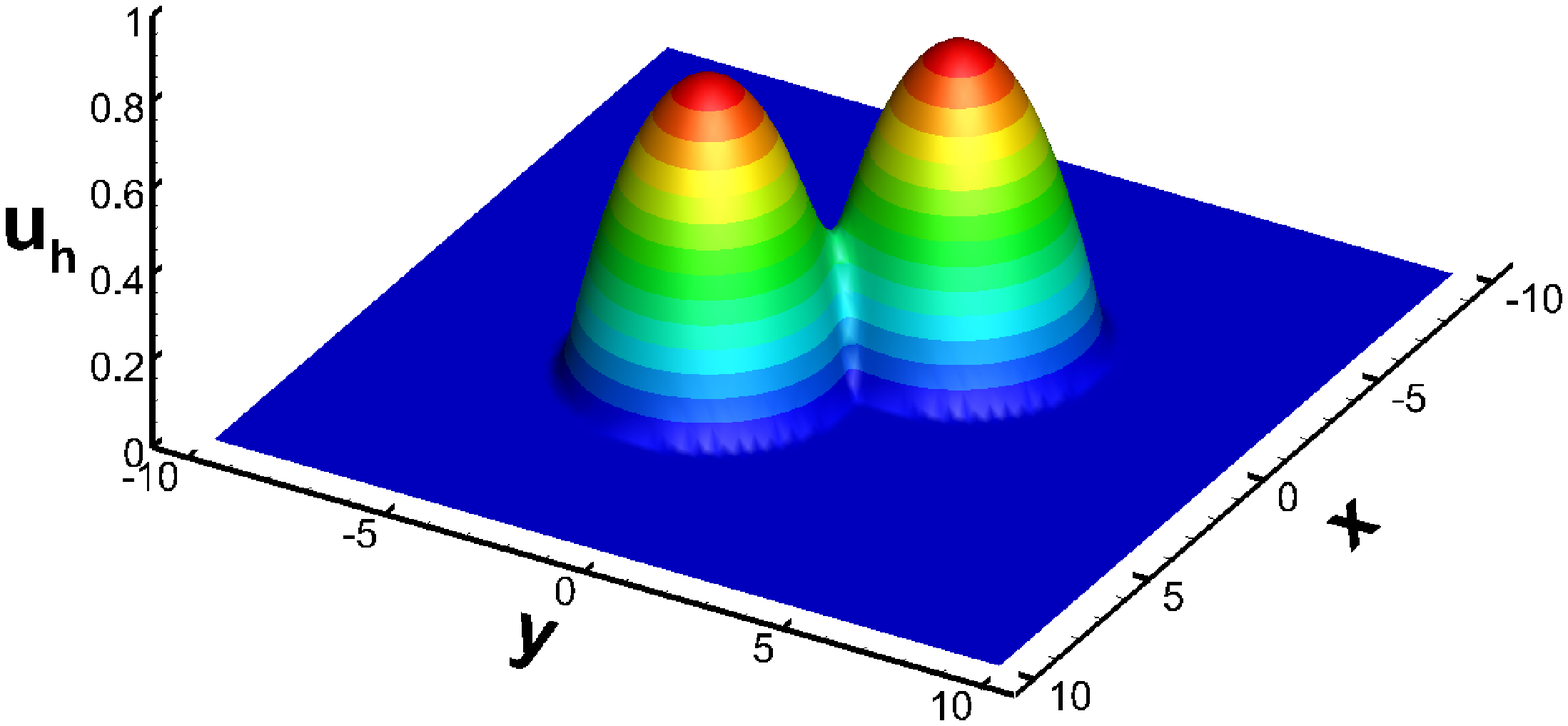} }
\subfigure[$t=1.$]{
\includegraphics[width=0.45\textwidth]{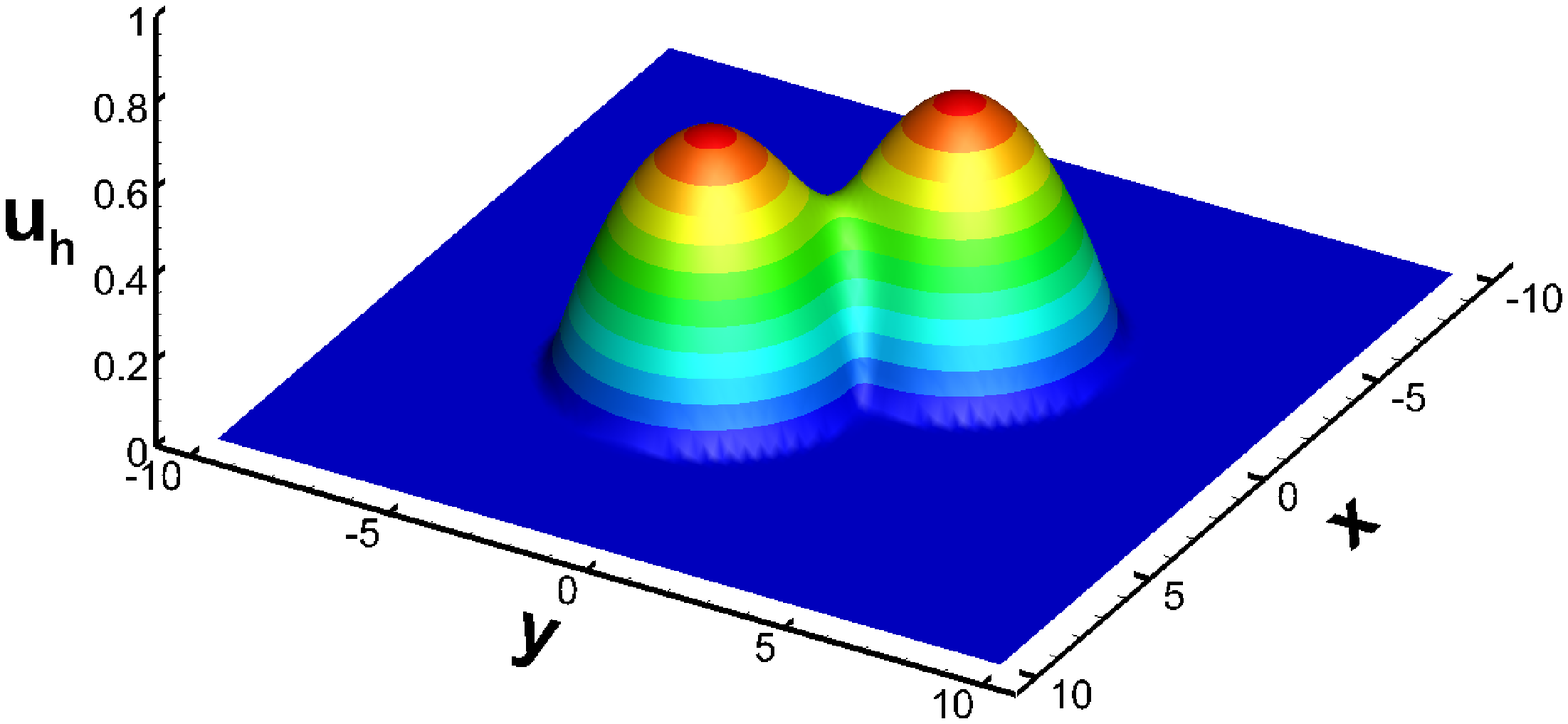} }
\subfigure[$t=4.$]{
\includegraphics[width=0.45\textwidth]{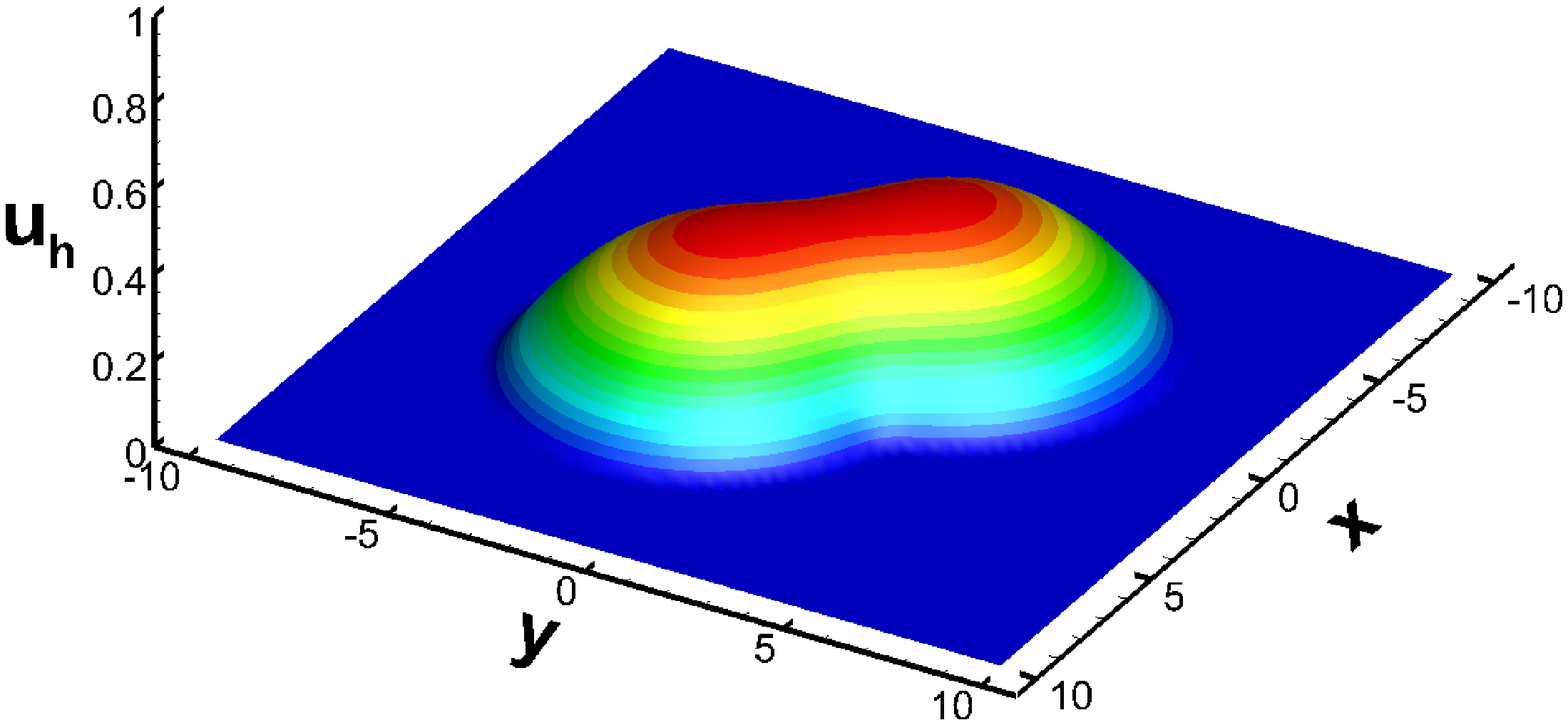} }
\end{figure}

\begin{figure}[htb]
\centering
\caption{ \label{fig_2d_SDP} Example \ref{examp6}: The 2D strongly parabolic equation.}
\subfigure[Contour]{
\includegraphics[width=0.45\textwidth]{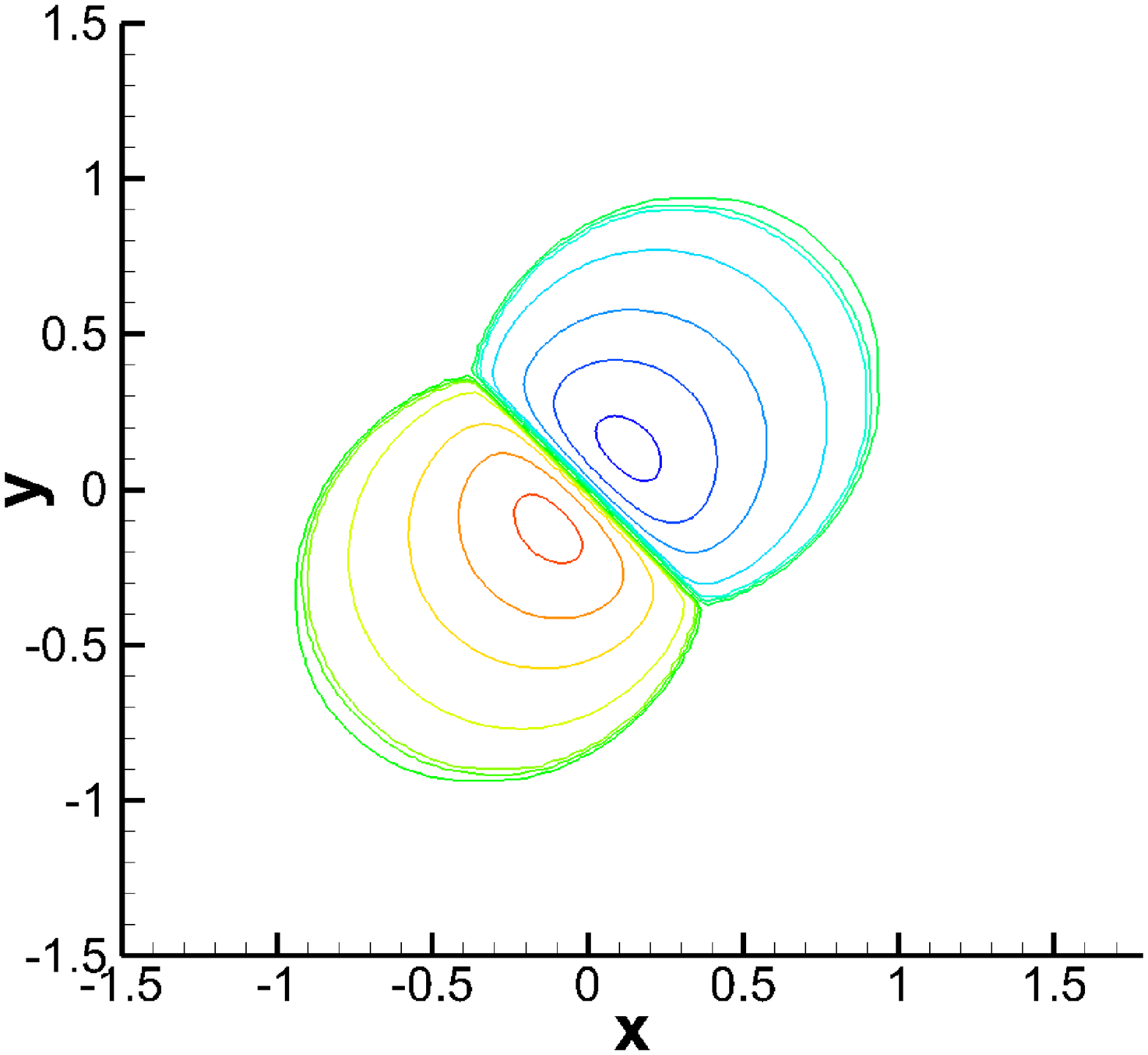} }
\subfigure[Surface]{
\includegraphics[width=0.45\textwidth]{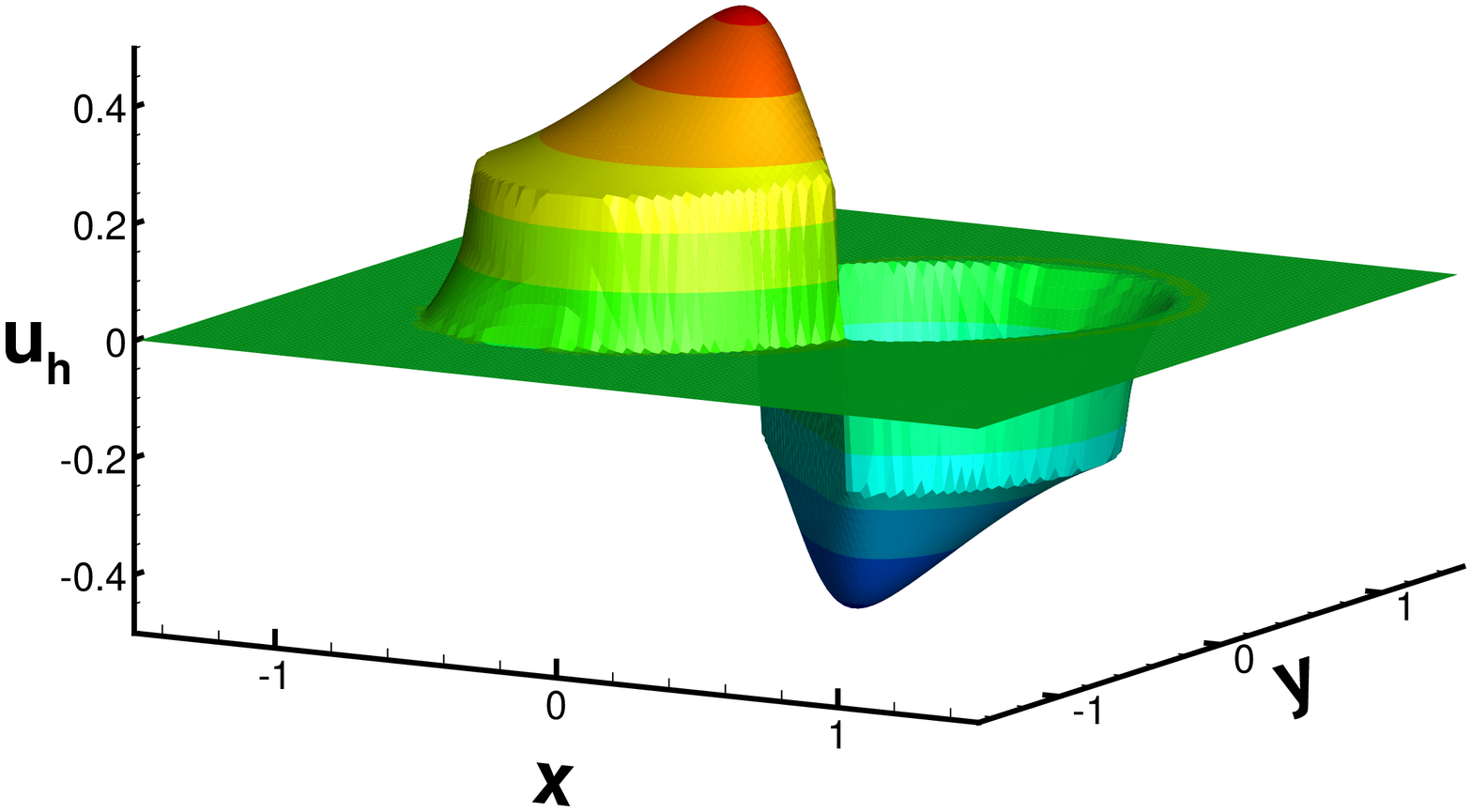} }
\end{figure}

\begin{exmp}\label{exmap5}
Next, we consider the two-dimensional PME
\begin{align}
u_t=(u^2)_{xx}+(u^2)_{yy},
\end{align}
in domain $[-10,10]\times[-10,10]$, with the initial condition
\begin{align}
u(x,y,0)=\left\{\begin{array}{ll}
\exp\left(\frac{-1}{6-(x-2)^2-(y+2)^2}\right),& (x-2)^2+(y+2)^2 <6,\\ \,
\exp\left(\frac{-1}{6-(x+2)^2-(y-2)^2}\right),& (x+2)^2+(y-2)^2 <6,\\
0, & \text{otherwise},
\end{array}\right.
\end{align}
and periodic boundary conditions in each directions.
\end{exmp}
The numerical solutions with $80\times 80$ uniform mesh at time $t=0,0.5,1.0$ and $4.0$ are shown in Figure \ref{fig_2d_PME}. The OFLDG scheme can capture the sharp interface without apparent oscillation.

\begin{exmp}\label{examp6}
Our final example is solving a two-dimensional strongly degenerate parabolic equation
\begin{align}
u_t+f(u)_x+f(u)_y = \epsilon (\nu(u)u_x)_x+\epsilon (\nu(u)u_y)_y,
\end{align}
on domain $[-1.5,1.5]\times[-1.5,1.5]$, where $f(u), \nu(u)$ and $\epsilon$ are the same as
the one-dimensional case in Example \ref{examp4}. The initial function is given as
\begin{align}
u(x,y,0)= \left\{\begin{array}{ll}
1,& (x+0.5)^2+(y+0.5)^2 <0.16,\\
-1,& (x-0.5)^2+(y-0.5)^2 <0.16,\\
0,& \text{otherwise.}
\end{array}\right.
\end{align}
\end{exmp}
The solution at $t=0.5$ computed by the OFLDG scheme with $120\times 120$ mesh cells is shown in Figure \ref{fig_2d_SDP}, which agrees well with the results in \cite{Jiang2021JSC,Liu2011SISC}.

\section{Concluding remarks}
\label{sec_sum}

In this paper, we propose a novel oscillation free local discontinuous Galerkin (OFLDG) method
to solve the nonlinear degenerate parabolic equations. This work is an extension of our recent work \cite{LLS2021SINUM}.
The key idea of the OFLDG method is to add some damping to the high order coefficients ($k \geq 1$).
The added damping terms not only preserve the high-order accuracy in smooth regions,
but also control the spurious oscillation well when the solution is of low regularity.
The $L^2$-stability and the optimal error estimates of semi-discrete schemes are rigorously established
for both one- and multidimensional nonlinear problems.
Several numerical examples are shown to demonstrate the effectiveness and robustness of the proposed scheme. Our next work is to extend the current framework to systems such as Navier-Stokes equations.

\end{document}